 \documentclass[11pt]{article}
\usepackage{amsfonts}
\usepackage{amsmath,amssymb}
\usepackage{amsthm}
\usepackage[mathscr]{euscript}
\usepackage{mathrsfs}
\usepackage{indentfirst}
\usepackage{color}
\usepackage{accents}
\usepackage{enumerate}

\newtheorem{theorem}{Theorem}[section]
\newtheorem{lemma}{Lemma}[section]

\newtheorem{remark}{Remark}[section]
\newtheorem{proposition}{Proposition}[section]
\numberwithin{equation}{section}

\allowdisplaybreaks

\newcommand{\be}{\begin{equation}}
\newcommand{\ee}{\end{equation}}
\newcommand\bes{\begin{eqnarray}} \newcommand\ees{\end{eqnarray}}
\newcommand{\bess}{\begin{eqnarray*}}
\newcommand{\eess}{\end{eqnarray*}}
\newcommand{\bbbb}{\left\{\begin{aligned}}
\newcommand{\nnnn}{\end{aligned}\right.}
\newcommand{\bea}{\begin{align*}}
\newcommand{\eea}{\end{align*}}

\newcommand\ep{\varepsilon}

\newcommand\kk{\left}
\newcommand\rr{\right}
\newcommand\dd{\displaystyle}

\newcommand\dx{{\rm d}x}
\newcommand\dy{{\rm d}y}

\newcommand\yy{\infty}

\begin{document}\thispagestyle{empty}
\setlength{\baselineskip}{16pt}
\begin{center}
 {\LARGE\bf The free boundary problem of an epidemic model with nonlocal diffusions and nonlocal reactions: spreading-vanishing dichotomy\footnote{This work was supported by NSFC Grants
12171120, 11901541,12301247}}\\[4mm]
{\Large Xueping Li}\\[0.5mm]
{School of Mathematics and Information Science, Zhengzhou University of Light Industry, Zhengzhou, 450002, China}\\[2.5mm]
  {\Large Lei Li}\\[0.5mm]
{College of Science, Henan University of Technology, Zhengzhou, 450001, China}\\[2.5mm]
{\Large Mingxin Wang\footnote{Corresponding author. {\sl E-mail}: mxwang@hpu.edu.cn}}\\[0.5mm]
 {School of Mathematics and Information Science, Henan Polytechnic University, Jiaozuo, 454000, China}
\end{center}

\date{\today}

\begin{quote}
\noindent{\bf Abstract.} This paper concerns the free boundary problem of an epidemic model. The spatial movements of the infectious agents and the infective humans are approximated by nonlocal diffusion operators. Especially, both the growth rate of the agents and the infective rate of humans are represented by nonlocal reaction terms. Thus our model has four integral terms which bring some difficulties for the study of the corresponding principal eigenvalue problem. Firstly, using some elementray analysis instead of Krein-Rutman theorem and the variational characteristic, we obtain the existence and asymptotic behaviors of principal eigenvalue.  Then a spreading-vanishing dichotomy is proved to hold, and the criteria for spreading and vanishing are derived. Lastly, comparing our results with those in the existing works, we discuss the effect of nonlocal reaction term on spreading and vanishing, finding that the more nonlocal reaction terms a model has, the harder spreading happens.

\textbf{Keywords}: Nonlocal diffusion; nonlocal reaction term; epidemic model; free boundary; principal eigenvalue; spreading and vanishing;

\textbf{AMS Subject Classification (2000)}: 35K57, 35R09,
35R20, 35R35, 92D25
\end{quote}

\section{Introduction}
\renewcommand{\thethm}{\Alph{thm}}
{\setlength\arraycolsep{2pt}
To understand the spreading mechanism of an oral-transmitted epidemic such as cholera, Capasso and Paveri-fontana \cite{CP} proposed the following ODE system
\bes \label{1.1}u_t=-au+cv, \;\;\; v_t=-bv+G(u),\; \;t>0,\ees
where $u$ stands for the average density of the infective agents, such as bacteria, virus and so on; $v$ represents the average density of the infective human population; $au$ denotes the natural death rate of the agents; $cv$ represents the growth rate of the agents contributed by the infective humans; $bv$ stands for the fatality rate of the infective human population. All these constants are positive. The function $G(u)$ represents the infective rate of humans, and satisfies
\begin{enumerate}
\item[{\bf(G)}] $G\in C^1([0,\yy))$, $G(0)=0$, $G'(z)>0$ for $z\ge 0$, $\frac{G(z)}{z}$ is strictly decreasing for $z>0$ and $\lim_{z\to\yy}\frac{G(z)}{z}<\frac{ab}{c}$.
 \end{enumerate}
 An example of $G$ is $\frac{\beta z}{1+\alpha z}$ with $\alpha,\beta>0$. The authors showed that the basic reproduction number
\[\mathcal{R}_0:=\frac{cG'(0)}{ab}\]
plays a crucial role in the dynamics of \eqref{1.1}. Namely, if $\mathcal{R}_0<1$, the disease-free equilibrium $(0,0)$ is asymptotically stable ( also exponentially stable by a simple comparison argument); while if $\mathcal{R}_0>1$, there exists a unique positive equilibrium $(u^*,v^*)$ which is asymptotically stable and determined by
\bes\label{1.2}\frac{G(u^*)}{u^*}=\frac{ab}{c},\;\; v^*=\frac{au^*}{c}.\ees

Over the past years, rapid progress on the research of model \eqref{1.1} has been made along many different directions. For example, please see \cite{CM} for the corresponding reaction-diffusion system on a bounded spatial domain, \cite{ZW,WSL} for traveling wave solutions and  \cite{ABL,WC,LZW,ZLN} for free boundary problems with random diffusion. Particularly, in \cite{CDLL}, the nonlocal diffusion operator
\[d\int_{\mathbb{R}}J(x-y)u(t,y)\dy-du\]
was first incorporated into free boundary problem arsing from ecology. Inspired by this work,  many studies have introduced such nonlocal diffusion operator and free boundary condition to \eqref{1.1}. Zhao et al \cite{ZZLD} proposed the following model
\bes\left\{\begin{aligned}\label{1.3}
&u_t=d\int_{g(t)}^{h(t)}\!\!J(x-y)u(t,y)\dy-du-au+cv, & &t>0,~x\in(g(t),h(t)),\\[1mm]
&v_t=-bv+G(u),& &t>0, ~ x\in(g(t),h(t)),\\[1mm]
&u(t,x)=v(t,x)=0,& &t>0, ~ x\notin(g(t),h(t))\\
&h'(t)=\mu\int_{g(t)}^{h(t)}\!\!\int_{h(t)}^{\infty}
J(x-y)u(t,x)\dy\dx,& &t>0,\\[1mm]
&g'(t)=-\mu\int_{g(t)}^{h(t)}\!\!\int_{-\infty}^{g(t)}
J(x-y)u(t,x)\dy\dx,& &t>0,\\[1mm]
&h(0)=-g(0)=h_0>0,\;\; u(0,x)=u_0(x); ~ ~ v(0,x)=v_0(x),& &|x|\le h_0,
 \end{aligned}\right.
 \ees
 where kernel function $J$ satisfies
 \begin{enumerate}
\item[{\bf(J)}]$J\in C(\mathbb{R})\cap L^{\yy}(\mathbb{R})$, $J(x)\ge0$, $J(0)>0$, $J$ is even, $\dd\int_{\mathbb{R}}J(x)\dx=1$.
 \end{enumerate}
 The initial data $(u_0,v_0)$ meet with
 \begin{enumerate}
\item[{\bf(H)}]$ \varrho\in C([-h_0,h_0])$, $\varrho(x)>0$ in $(-h_0,h_0)$, $\varrho(\pm h_0)=0$.
 \end{enumerate}
 They proved that the dynamics of \eqref{1.3} is govern by a spreading-vanishing dichotomy.  The criteria for spreading and vanishing were also obtained by using a related principal eigenvalue problem and some comparison principles. The spreading speed was given by Du and Ni \cite{DN1} where a systematic method for monostable cooperative systems was put forward by following the lines in \cite{DLZJMPA}.

Subsequently, since the infectious agents not only depend on the infective humans at location $x$, but on some neighbourhood of $x$, Zhao et al \cite{ZLD} replaced the term $cv$ in \eqref{1.3} by the nonlocal reaction term $c\int_{g(t)}^{h(t)}K(x-y)v(t,y)\dy$. More precisely, they studied the problem
 \bes\left\{\begin{aligned}\label{1.4}
&u_t=d\int_{g(t)}^{h(t)}\!\!J(x-y)u(t,y)\dy-du-au+c\int_{g(t)}^{h(t)}K(x-y)v(t,y)\dy, & &t>0,~x\in(g(t),h(t)),\\[1mm]
&v_t=-bv+G(u),& &t>0, ~ x\in(g(t),h(t)),\\[1mm]
&u(t,x)=v(t,x)=0,& &t>0, ~ x\notin(g(t),h(t))\\
&h'(t)=\mu\int_{g(t)}^{h(t)}\!\!\int_{h(t)}^{\infty}
J(x-y)u(t,x)\dy\dx,& &t>0,\\[1mm]
&g'(t)=-\mu\int_{g(t)}^{h(t)}\!\!\int_{-\infty}^{g(t)}
J(x-y)u(t,x)\dy\dx,& &t>0,\\[1mm]
&h(0)=-g(0)=h_0>0,\;\; u(0,x)=u_0(x); ~ ~ v(0,x)=v_0(x),& &|x|\le h_0,
 \end{aligned}\right.
 \ees
 where kernel functions $J$ and $K$ satisfy {\bf (J)}.
 The authors found that similar to \eqref{1.3}, the dynamics of \eqref{1.4} also conforms to a spreading-vanishing dichotomy, but the criteria for spreading and vanishing are different from that of \eqref{1.3}. The spreading speed was investigated in another work \cite{DLNZ} where one can see that the spreading speed of the corresponding Cauchy problem with compact initial data is finite if and only if there is some $\lambda>0$ such that both $J$ and $K$ meet with $\int_{\mathbb{R}}P(x)e^{-\lambda x}\dx<\yy$ with $P=J$ or $K$, while the spreading speed of \eqref{1.4} is finite if and only if $J$ satifies $\int_{0}^{\yy}xJ(x)\dx<\yy$ without extra assumption on $K$.

 As is seen in models \eqref{1.3} and \eqref{1.4}, the dispersal of the infective humans $v$ is ignored by assuming that the mobility of the infective humans is relatively small compared to the agents $u$. However, Chang and Du \cite{CD} supposed that the diffusion of $v$ is also approximated by nonlocal diffusion operator, and thus proposed the following model
  \bes\left\{\begin{aligned}\label{1.5}
&u_t=d_1\int_{g(t)}^{h(t)}\!\!J_1(x-y)u(t,y)\dy-d_1u-au+cv, & &t>0,~x\in(g(t),h(t)),\\[1mm]
&v_t=d_2\int_{g(t)}^{h(t)}\!\!J_2(x-y)v(t,y)\dy-d_2v-bv+G(u),& &t>0, ~ x\in(g(t),h(t)),\\[1mm]
&u(t,x)=v(t,x)=0,& &t>0, ~ x\notin(g(t),h(t))\\
&h'(t)=\int_{g(t)}^{h(t)}\!\!\int_{h(t)}^{\infty}\bigg[\mu_1
J_1(x-y)u(t,x)+\mu_2J_2(x-y)v(t,x)\bigg]\dy\dx,& &t>0,\\[1mm]
&g'(t)=-\int_{g(t)}^{h(t)}\!\!\int_{-\yy}^{g(t)}\bigg[\mu_1
J_1(x-y)u(t,x)+\mu_2J_2(x-y)v(t,x)\bigg]\dy\dx,& &t>0,\\[1mm]
&h(0)=-g(0)=h_0>0,\;\; u(0,x)=u_0(x); ~ ~ v(0,x)=v_0(x),& &|x|\le h_0,
 \end{aligned}\right.
 \ees
 where $J_1$ and $J_2$ satisfy {\bf (J)}. It was proved in \cite{CD} that the dynamics of \eqref{1.5} is similar to those of the above models \eqref{1.3} and \eqref{1.4}, namely, a spreading-vanishing dichotomy holds. The spreading speed was covered in \cite{DN1}. Afterwards, as in \eqref{1.4}, Du and Wang \cite{WD1} changed the term $cv$ to the nonlocal reaction term $c\int_{g(t)}^{h(t)}K(x-y)v(t,y)\dy$ and stduied the problem
   \bes\left\{\begin{aligned}\label{1.6}
&u_t=d_1\int_{g(t)}^{h(t)}\!\!J_1(x-y)u(t,y)\dy-d_1u-au+c\int_{g(t)}^{h(t)}K(x-y)v(t,y)\dy, & &t>0,~x\in(g(t),h(t)),\\[1mm]
&v_t=d_2\int_{g(t)}^{h(t)}\!\!J_2(x-y)v(t,y)\dy-d_2v-bv+G(u),& &t>0, ~ x\in(g(t),h(t)),\\[1mm]
&u(t,x)=v(t,x)=0,& &t>0, ~ x\notin(g(t),h(t))\\
&h'(t)=\int_{g(t)}^{h(t)}\!\!\int_{h(t)}^{\infty}\bigg[\mu
J_1(x-y)u(t,x)+\mu\rho J_2(x-y)v(t,x)\bigg]\dy\dx,& &t>0,\\[1mm]
&g'(t)=-\int_{g(t)}^{h(t)}\!\!\int_{-\yy}^{g(t)}\bigg[\mu
J_1(x-y)u(t,x)+\mu \rho J_2(x-y)v(t,x)\bigg]\dy\dx,& &t>0,\\[1mm]
&h(0)=-g(0)=h_0>0,\;\; u(0,x)=u_0(x); ~ ~ v(0,x)=v_0(x),& &|x|\le h_0,
 \end{aligned}\right.
 \ees
 in which $J_1$, $J_2$ and $K$ satisfy {\bf (J)}. They proved that the spreading-vanishing dichotomy also holds for \eqref{1.6}, and the spreading speed was shown in their another work \cite{WD2}.

In view of the above models, it is natural to think that the nonlocal diffusion of the infectious agents may let the infective rate of humans not only to depend on the agents at location $x$, but on some neighborhood of $x$. Hence we change the infective rate $G(u)$ to the nonlocal reaction term
 \[G(\int_{g(t)}^{h(t)}J_{21}(x-y)u(t,y)\dy),\]
and consider the following problem
  \bes\left\{\begin{aligned}\label{1.7}
&u_t=d_1\int_{g(t)}^{h(t)}\!\!J_{11}(x-y)u(t,y)\dy-d_1u-au+c\int_{g(t)}^{h(t)}J_{12}(x-y)v(t,y)\dy, & &t>0,x\in(g(t),h(t)),\\[1mm]
&v_t=d_2\int_{g(t)}^{h(t)}\!\!J_{22}(x-y)v(t,y)\dy-d_2v-bv+G(\int_{g(t)}^{h(t)}J_{21}(x-y)u(t,y)\dy),& &t>0, x\in(g(t),h(t)),\\[1mm]
&u(t,x)=v(t,x)=0,& &t>0, ~ x\notin(g(t),h(t))\\
&h'(t)=\int_{g(t)}^{h(t)}\!\!\int_{h(t)}^{\infty}\bigg[\mu_1
J_1(x-y)u(t,x)+\mu_2 J_2(x-y)v(t,x)\bigg]\dy\dx,& &t>0,\\[1mm]
&g'(t)=-\int_{g(t)}^{h(t)}\!\!\int_{-\yy}^{g(t)}\bigg[\mu_1
J_1(x-y)u(t,x)+\mu_2J_2(x-y)v(t,x)\bigg]\dy\dx,& &t>0,\\[1mm]
&h(0)=-g(0)=h_0>0,\;\; u(0,x)=u_0(x); ~ ~ v(0,x)=v_0(x),& &|x|\le h_0,
 \end{aligned}\right.
 \ees
 where condtions {\bf (J)} and {\bf(H)} hold for kernel functions $J_{ij}$ with $i,j=1,2$  and $(u_0,v_0)$, respectively.
  With the aid of a rahter complete undertanding for the asymptotic behaviors of a related principal eigenvalue, we aim at knowing as much as possible about the dynamics of \eqref{1.7} in the present work. The spreading speed and rate of accelerated spreading will be discussed in another work. Below are our main results in this paper.

  \begin{theorem}[Global existence and uniqueness]\label{t1.1}Problem \eqref{1.7} has a unique global solution $(u,v,g,h)$. Moreover, $(u,v)\in [C([0,\yy)\times[g(t),h(t)])]^2$, $(g,h)\in [C^1([0,\yy))]^2$, $0< u(t,x)\le K_1$ and $0< v(t,x)\le K_2$ in $[0,\yy)\times(g(t),h(t))$ with some $K_1,K_2>0$ depending only on the initial data $(u_0,v_0)$ and parameters of \eqref{1.7}.
\end{theorem}

It follows from the above theorem and the equations of $g$ and $h$ that $-g$ and $h$ are strictly increasing in $t>0$. Thus $\dd\lim_{t\to\yy}h(t)=:h_{\yy}$ and $\lim_{t\to\yy}g(t)=:g_{\yy}$ are both well defined, as well as $h_{\yy}\in(h_0,\yy]$, $g_{\yy}\in[-\yy,-h_0)$.

\begin{theorem}[Spreading-vanishing dichotomy]\label{t1.2} Let $(u,v,g,h)$ be the unique solution of \eqref{1.7}. Then one of the following alternatives must happen.
\begin{enumerate}[$(1)$]
\item \underline{Spreading:} necessarily $\mathcal{R}_0>1$, $\dd\lim_{t\to\yy}g(t)=-\yy$, $\dd\lim_{t\to\yy}h(t)=\yy$, $\dd\lim_{t\to\yy}u(t,x)=u^*$ and $\dd\lim_{t\to\yy}v(t,x)=v^*$ in $C_{loc}(\mathbb{R})$, where $(u^*,v^*)$ is unique given by \eqref{1.2}.

\item \underline{Vanishing:} $h_{\yy}-g_{\yy}<\yy$ and $\dd\lim_{t\to\yy}\|u(t,\cdot)+v(t,\cdot)\|_{C([g(t),h(t)])}=0$. Moreover, if $\lambda_p(g_{\yy},h_{\yy})<0$ (easily guaranteed by letting $\mathcal{R}_0\le1$ or $\mu_1+\mu_2$ small enough), then $\dd\lim_{t\to\yy}e^{kt}\|u(t,\cdot)+v(t,\cdot)\|_{C([g(t),h(t)])}=0$ for all $k\in(0,-\lambda_p(g_{\yy},h_{\yy}))$, where $\lambda_p(g_{\yy},h_{\yy})$ is determined by principal eigenvalue problem \eqref{2.9}.
    \end{enumerate}
\end{theorem}

By virtue of a related principal eigenvalue problem and some comparison arguments, we obtain the following criteria for spreading and vanishing.
\begin{theorem}[Criteria for spreading and vanishing]\label{t1.3} Let $(u,v,g,h)$ be the unique solution of \eqref{1.7}.
\begin{enumerate}[$(1)$]
\item If $\mathcal{R}_0\le1$, vanishing happens.
\item Suppose $\mathcal{R}_0>1$. Then there exists a unique critical length $L^*$ for initial habitat $[-h_0,h_0]$ such that spreading happens if $2h_0\ge L^*$, where $L^*$ is uniquely determined by principal eigenvalue problem \eqref{2.9}.
\item Assume that $\mathcal{R}_0>1$, $2h_0<L^*$ and $\mu_2=f(\mu_1)$ with $f\in C([0,\yy))$, $f(0)=0$ and strictly increasing to $\yy$. Then there exists a unique $\mu^*_1$ such that spreading happens if and only if $\mu_1>\mu^*_1$.
\item Let $\mathcal{R}_0>1$, $2h_0<L^*$ and $J_{ii}>0$ in $\mathbb{R}$ for $i=1,2$. We parameterize the initial data $(u_0,v_0)=(\tau\vartheta_1,\tau\vartheta_2)$ with $\tau>0$ and $(\vartheta_1,\vartheta_2)$ satisfying {\bf(H)}. Then there exists a unique $\tau^*$ such that spreading happens if and only if $\tau>\tau^*$.
\item Assume that $\mathcal{R}_0>1$, $c=G'(0)$, $J_{12}=J_{21}$ and $d_2=f(d_1)$ with $f$ defined as above. Then there exists a unique $\tilde{L}^*<L^*$ depending only on $(a,b,c,J_{12})$ such if $2h_0>\tilde{L}^*$, then we can find a unique $d^*_1$ such that spreading happens if $d_1\le d^*$, while if $d_1>d^*_1$, whether spreading or vanishing occurs depends on $\mu_i$ for $i=1,2$ as in {\rm (3)}. If $2h_0\le \tilde{L}^*$, then for any $d_1>0$, both spreading and vanishing may happen, depending on $\mu_i$ for $i=1,2$ as in {\rm (3)}.
\end{enumerate}
\end{theorem}

This paper is arranged as follows. Section 2 is devoted to some preliminary works, involving a principal eigenvalue problem, a steady state problem and a fixed boundary problem. Especially, a relatively complete understanding for asymptotic behaviors of principal eigenvalue is given. Section 3 deals with the dynamics of \eqref{1.7} by using the results in Section 2 and some comparison arguments. Section 4 concerns a brief discussion about the effect of nonlocal reaction term on spreading and vanishing. We compare the principal eigenvalues associated to models \eqref{1.3}-\eqref{1.7}, and find a relationship between their critical lengths of initial habitat which implies that the more nonlocal reaction terms a model has, the harder spreading happens for this model.

Before ending the introduction, we would like to mention that compared to the free boundary problem with random diffusion as in \cite{DL} whose spreading speed is always finite, the counterpart with nonlocal diffusion as in \cite{CDLL} can have an infinite spreading speed if kernel function violates a threshold condition (see \cite{DLZJMPA}). In recent years, many works concerning the free boundary problem with nonlocal diffusion have emerged along different directions. For example, one can refer to \cite{CQW,DN2,DN3,DN4,NV,PLL,DWZ,DN5,ZLZ,LLW1} and the references therein.

 \section{Some preliminary works}
This section involves a principal eigenvalue problem, a steady state problem  and a fixed boundary problem. The understanding for such problems will pave the road for our later discussion for \eqref{1.7}. Let us begin with studying the principal eigenvalue problem associated to \eqref{1.7}. Our ideas come from \cite{SWZ,LCW,HMMV,SLLW,Zhanglei}. We would like to mention that in \cite{WD1}, they study the corresponding principal eigenvalue problem by resorting to a perturbation method (see \cite{Bu}) and Krein-Rutman theorem.

Suppose that $a_{12}$, $a_{21}$, $b_1$, $b_2$ are positive, $a_{ii}\ge0$, and $l_1,l_2\in\mathbb{R}$.  Define the nonlocal operator
\[\mathcal{L}[\varphi](x):=\mathcal{P}[\varphi](x)+H\varphi(x), ~ ~x\in[l_1,l_2],\]
where $\varphi=(\varphi_1(x),\varphi_2(x))^T$,
\[\mathcal{P}[\varphi](x)=\begin{pmatrix}
                                             \dd a_{11}\int_{l_1}^{l_2}J_{11}(x-y)\varphi_1(y)\dy+ a_{12}\int_{l_1}^{l_2}J_{12}(x-y)\varphi_2(y)\dy\\
                                              \dd a_{21}\int_{l_1}^{l_2}J_{21}(x-y)\varphi_1(y)\dy+ a_{22}\int_{l_1}^{l_2}J_{22}(x-y)\varphi_2(y)\dy
                                           \end{pmatrix}, ~ ~ H=\begin{pmatrix}
                                                                     -b_1 & 0 \\
                                                                     0 & -b_2
                                                                 \end{pmatrix}.\]
Thus $\mathcal{L}[\varphi]=\lambda\varphi$ exactly takes the form of
 \bes\left\{\begin{aligned}\label{2.1}
&a_{11}\int_{l_1}^{l_2}J_{11}(x-y)\varphi_1(y)\dy+a_{12}\int_{l_1}^{l_2}J_{12}(x-y)\varphi_2(y)\dy-b_{1}\varphi_1=\lambda\varphi_1, & &x\in[l_1,l_2],\\[1mm]
&a_{21}\int_{l_1}^{l_2}J_{21}(x-y)\varphi_1(y)\dy+a_{22}\int_{l_1}^{l_2}J_{22}(x-y)\varphi_2(y)\dy-b_{2}\varphi_2=\lambda\varphi_2, & &x\in[l_1,l_2].
 \end{aligned}\right.
 \ees
For clarity, we make the following denotations
\bess
&&X=[C([l_1,l_2])]^2, ~ ~ X^+=\{\varphi\in X: \varphi_1\ge0,~ \varphi_2\ge0\}, ~ ~X^{++}=\{\varphi\in X^+: \varphi_1>0,~\varphi_2>0\}, \\
&&E=[L^2([l_1,l_2])]^2, ~ ~ \big\langle\varphi,\psi\big\rangle=\sum_{i=1}^{2}\int_{l_1}^{l_2}\varphi_i(x)\psi_i(x)\dx, ~ ~ \|\varphi\|_2=\sqrt{\big\langle\varphi,\varphi\big\rangle}, ~ ~ {\bf 0}=(0,0),\\
&&\mathcal{L}_i[\varphi]=a_{i1}\int_{l_1}^{l_2}J_{i1}(x-y)\varphi_1(y)\dy+ a_{i2}\int_{l_1}^{l_2}J_{i2}(x-y)\varphi_2(y)\dy-b_i\varphi_i.
\eess
Define
\[\lambda_p=\inf\{\lambda\in\mathbb{R}: \mathcal{L}[\varphi]\le\lambda\varphi ~ {\rm for ~ some ~}\varphi\in X^{++}\}.\]
Clearly, $\lambda_p$ is equivalent to
\[\lambda_p=\inf_{\varphi\in X^{++}}\sup_{x\in[l_1,l_2], ~ i=1,2}\frac{\mathcal{L}_i[\varphi](x)}{\varphi_i(x)}.\]

It is well known that $\lambda$ is a principal eigenvalue if it is simple and its eigenfunction $\varphi\in X^{++}$.  The next Strong Maximum Principle and Touching Lemma are curcial for proving the existence of the principal eigenvalue of \eqref{2.1}.

 \begin{lemma}[Strong Maximum Principle]\label{l2.1}If $\varphi\in X^{+}\setminus\{{\bf0}\}$ satisfies $\mathcal{L}[\varphi]\le{\bf0}$, then $\varphi\in X^{++}$.
 \end{lemma}
 \begin{proof}
 Assume on the contrary that $\dd\min\{\min_{x\in[l_1,l_2]}\varphi_1(x),\min_{x\in[l_1,l_2]}\varphi_1(x)\}=0$. Then there exists a $x_0\in[l_1,l_2]$ such that $\varphi_1(x_0)=0$ or $\varphi_2(x_0)=0$. Without loss of generality, we suppose that $\varphi_1(x_0)=0$. If $\varphi_1(x)\equiv0$ in $[l_1,l_2]$, then it is easy to see $\varphi_2(x)\equiv0$ in $[l_1,l_2]$ which contradicts the assumption $\varphi\in X^{+}\setminus\{{\bf0}\}$. Thus $\varphi_1(x)\not\equiv0$ in $[l_1,l_2]$. By continuity, there exists a $x_1\in[l_1,l_2]$ such that $\varphi_1(x_1)=0$ and $\varphi_1(x)>0$ in some left or right small neighborhood of $x_1$. If $a_{11}>0$, then substituting $x_1$ into $\mathcal{L}_1[\varphi]\le0$ leads to
 \[0<a_{11}\int_{l_1}^{l_2}J_{11}(x_1-y)\varphi_1(y)\dy+a_{12}\int_{l_1}^{l_2}J_{12}(x_1-y)\varphi_2(y)\dy\le0.\]
 This is a contradiction.

 If $a_{11}=0$, we have
 \[\int_{l_1}^{l_2}J_{12}(x_1-y)\varphi_2(y)\dy=0,\]
 which, by condition {\bf (J)}, implies that $\varphi_2(x)=0$ somewhere on $[l_1,l_2]$. If $a_{22}>0$, then a similar contradiction can be derived. If $a_{22}=0$, recalling $\varphi_2\not\equiv0$ in $[l_1,l_2]$,  we can choose a $x_2\in[l_1,l_2]$ such that $\varphi_2(x_2)=0$ and $\varphi_2(x)>0$ in some left or right small neighborhood of $x_2$. Taking $x_2$ into $\mathcal{L}_1[\varphi]\le0$ yields $\varphi_1(x_2)>0$. Then substituting $x_2$ into $\mathcal{L}_2[\varphi]\le0$ leads to a contradiction. Therefore, $\varphi\in X^{++}$.
 The proof is complete.
 \end{proof}

 \begin{lemma}[Touching Lemma]\label{l2.2}Assume that there exists a $\varphi\in X^{++}$ such that $\mathcal{L}[\varphi]\le {\bf0}$. If $\mathcal{L}[\psi]\le{\bf0}$ for some $\psi\in X$, then either $\psi\ge{\bf0}$, or $\psi$ is a negative constant multiple of $\varphi$ and $\mathcal{L}[\varphi]=\mathcal{L}[\psi]={\bf0}$.
 \end{lemma}
 \begin{proof}If $\psi\in X^+$, we have nothing to prove. If $\psi\not\in X^+$, then $\psi_1(x)$ or $\psi_2(x)$ must be negative somewhere. Define $\phi=\varphi+\delta\psi$ for $\delta>0$. It is easy to see that $\phi\in X^{++}$ for sufficiently small $\delta>0$, while $\phi_1(x)$ or $\phi_2(x)$ is negative somewhere if $\delta$ is large enough. By continuity, there exists a smallest $\delta^*>0$ such that $\phi\in X^+$ and $\phi_1(x)$ or $\phi_2(x)$ vanishes somewhere. Due to $\mathcal{L}[\phi]\le{\bf 0}$, by Lemma \ref{l2.1}, we see $\phi={\bf0}$ and $\psi=-\frac{1}{\delta^*}\varphi$. Moreover, $\mathcal{L}[\varphi]=\mathcal{L}[\psi]={\bf0}$. The proof is ended.
 \end{proof}
\begin{proposition}\label{p2.1} Let $\lambda_p$ be defined as above. The following statements are valid.
 \begin{enumerate}[$(1)$]
 \item $\lambda_{p}$ is an eigenvalue of operator $\mathcal{L}$ with a corresponding eigenfunction $\varphi_p\in X^{++}$.

 \item The algebraic multiplicity of $\lambda_p$ is equal to one, which implies $\lambda_p$ is simple.

 \item If there exists an eigenpair $(\lambda,\varphi)$ with $\varphi\in X^+\setminus\{{\bf0}\}$, then $\lambda=\lambda_p$ and $\varphi$ is a positive constant multiple of $\varphi_p$.

 \item If $a_{12}=a_{21}$ and $J_{12}=J_{21}$, then we have the following variational characteristic
 \[\lambda_p=\sup_{\|\varphi\|_2=1}\langle\mathcal{L}[\varphi],\varphi\rangle.\]
 \end{enumerate}
\end{proposition}
\begin{proof} (1) {\bf Step 1.} We now prove that $\lambda_p$ is an eigenvalue with an eigenfunction $\varphi\in X^{++}$ if and only if $\lambda_p>\max\{-b_1,-b_2\}$. Obviously, $\lambda_p\ge \max\{-b_1,-b_2\}$. If $(\lambda_p,\varphi)$ with $\varphi\in X^{++}$ is an eigenpair of $\mathcal{L}$, then for $i=1,2$, we have
\[a_{i1}\int_{l_1}^{l_2}J_{i1}(x-y)\varphi_i(y)\dy+a_{i2}\int_{l_1}^{l_2}J_{i2}(x-y)\varphi_i(y)\dy=(b_{i}+\lambda)\varphi_i ~ ~ {\rm in ~ }[l_1,l_2],\]
which clearly implies $\lambda_p>\max\{-b_1,-b_2\}$.

Suppose that $\lambda_p>\max\{-b_1,-b_2\}$. Next we show there exists a $\varphi_p\in X^{++}$ such that $\mathcal{L}[\varphi_p]=\lambda_p\varphi_p$.

{\bf Claim 1.} For any $\lambda>\lambda_p$, the operator $\mathcal{L}-\lambda\mathcal{I}$ is a bijection.

Due to $\lambda>\max\{-b_1,-b_2\}$, the operator $\mathcal{H}-\lambda\mathcal{I}$ is a bijection, and thus is a Fredholm operator with index zero. Moreover, in view of Arzel{\'a}-Ascoli Theorem, operator $\mathcal{P}$ is compact. Hence $\mathcal{L}-\lambda\mathcal{I}=\mathcal{P}+\mathcal{H}-\lambda\mathcal{I}$ is a Fredholm operator with the same index as $\mathcal{H}-\lambda\mathcal{I}$. So to prove Claim 1, it is sufficient to show that $\mathcal{L}-\lambda\mathcal{I}$ is one-to-one, i.e. if $\varphi\in X$ satisfies $(\mathcal{L}-\lambda\mathcal{I})[\varphi]={\bf0}$, then $\varphi={\bf0}$. Without loss of generality, assume on the contrary that $\varphi_1<0$ somewhere on $[l_1,l_2]$. According to the definition of $\lambda_p$, there exists a $\varphi_{\lambda}\in X^{++}$ such that
\[\lambda>\sup_{x\in[l_1,l_2],i=1,2}\frac{\mathcal{L}_i[\varphi_{\lambda}](x)}{(\varphi_{\lambda}(x))_i},\]
which implies
\bes\label{2.2}
-(\mathcal{L}-\lambda\mathcal{I})[\varphi_{\lambda}]\in X^{++}.
\ees
However, applying Touching Lemma (Lemma \ref{l2.2}) with $(\mathcal{L}-\lambda\mathcal{I},\varphi_{\lambda},\varphi)$ in place of $(\mathcal{L}, \varphi, \psi)$, respectively, we derive $(\mathcal{L}-\lambda\mathcal{I})[\varphi_{\lambda}]={\bf0}$, which contradicts \eqref{2.2}. Therefore, Claim 1 is proved.

By Claim 1, for any $\delta>0$ and $\lambda\in(\lambda_p,\lambda_p+1]$, there exists a unique $\varphi^{\delta}_{\lambda}\in X$ such that
\bes
\label{2.3}(\mathcal{L}-\lambda\mathcal{I}){\varphi^{\delta}_{\lambda}}=-(\delta,\delta) ~ ~ \forall x\in[l_1,l_2].
\ees

{\bf Claim 2.} $\varphi^{\delta}_{\lambda}\in X^{++}$.

We first show $\varphi^{\delta}_{\lambda}\in X^{+}$. Otherwise, arguing as above and using Touching Lemma with $(\mathcal{L},\varphi,\psi)$ replaced by $(\mathcal{L}-\lambda\mathcal{I},\varphi_{\lambda},\varphi^{\delta}_{\lambda})$, we have $(\mathcal{L}-\lambda\mathcal{I})[\varphi_{\lambda}]={\bf0}$, which contradicts \eqref{2.2}. Thus $\varphi^{\delta}_{\lambda}\in X^+$. Then by the Strong Maximum principle (Lemma \ref{l2.1}), Claim 2 is verified.

Based on the above discussions, there are two cases to consider:\\
{\bf Case 1.} There exists a $\delta\in(0,1)$ such that for any $\lambda\in(\lambda_p,\lambda_p+1]$, we have $\int_{l_1}^{l_2}[(\varphi^{\delta}_{\lambda})_1(x)+(\varphi^{\delta}_{\lambda})_2(x)]\dx<1$;\\
{\bf Case 2.} For any $\delta\in(0,1)$, there exists $\lambda_{\delta}\in(\lambda_p,\lambda_p+1]$ such that $\int_{l_1}^{l_2}[(\varphi^{\delta}_{\lambda})_1(x)+(\varphi^{\delta}_{\lambda})_2(x)]\dx\ge1$.

We now show that Case 1 can not happen. Since $\varphi^{\delta}_{\lambda}\in X^{++}$ and $\int_{l_1}^{l_2}[(\varphi^{\delta}_{\lambda})_1(x)+(\varphi^{\delta}_{\lambda})_2(x)]<1$, it is easy to verify that $\{\mathcal{P}[\varphi^{\delta}_{\lambda}]:\lambda\in(\lambda_p,\lambda_p+1]\}$ is uniformly bounded and equicontinuous on $[l_1,l_2]$. By Arzel{\'a}-Ascoli Theorem, there exist a subsequence of $\lambda$ decreasing to $\lambda_p$ and $\omega(x)\in X^+$ such that $\mathcal{P}[\varphi^{\delta}_{\lambda}]\to \omega$ in $X$ as $\lambda\searrow\lambda_p$. By \eqref{2.3}, we see
\[\varphi^{\delta}_{\lambda}=(\lambda\mathcal{I}-\mathcal{H})^{-1}[\mathcal{P}[\varphi^{\delta}_{\lambda}]+(\delta,\delta)^T].\]
Letting $\lambda\to\lambda_p$ yields
\[\varphi^{\delta}_{\lambda}\to(\lambda_p\mathcal{I}-\mathcal{H})^{-1}[\omega+(\delta,\delta)^T]=:\varphi^{\delta} ~ ~{\rm in ~ }X.\]
Thus
\[\varphi^{\delta}=(\lambda_p\mathcal{I}-\mathcal{H})^{-1}[\mathcal{P}[\varphi^{\delta}]+(\delta,\delta)^T] ~ ~ {\rm and ~ ~ }-(\mathcal{H}-\lambda_p\mathcal{I})[\varphi^{\delta}]\in X^{++}.\]
Recall that $\lambda_p>\max\{-b_1,-b_2\}$. So $\varphi^{\delta}\in X^{++}$. Then using the definition of $\lambda_p$, we have
\bess
\lambda_p\le \sup_{x\in[l_1,l_2],i=1,2}\frac{\mathcal{L}_i[\varphi^{\delta}](x)}{\varphi^{\delta}_i}=\lambda_p-\inf_{x\in[l_1,l_2],i=1,2}\frac{\delta}{\varphi^{\delta}_i(x)}<\lambda_p,
\eess
which indicates that Case 1 cannot happen. Thus Case 2 holds true.

Define
\bess
\phi^{\delta}_i=\frac{(\varphi^{\delta}_{\lambda_{\delta}}(x))_i}{\dd\int_{l_1}^{l_2}[(\varphi^{\delta}_{\lambda_{\delta}})_1(x)+(\varphi^{\delta}_{\lambda_{\delta}})_2(x)]\dx}
, ~ ~ c_{\delta}=\frac{-\delta}{\dd\int_{l_1}^{l_2}[(\varphi^{\delta}_{\lambda_{\delta}})_1(x)+(\varphi^{\delta}_{\lambda_{\delta}})_2(x)]\dx}.
\eess
Clearly, $\int_{l_1}^{l_2}[\phi^{\delta}_1(x)+\phi^{\delta}_2(x)]\dx=1$, $c_{\delta}\to0$ as $\delta\to0^+$ and $(\mathcal{L}-\lambda\mathcal{I})[\phi^{\delta}]=(c_{\delta},c_{\delta})^T$. In view of a compact consideration, by passing a subsequence if necessary, we have $\lambda_{p}\to\lambda_0\in[\lambda_p,\lambda_p+1]$ and $\mathcal{P}[\phi^{\delta}]$ converges in $X$. Since
\[\phi^{\delta}=(\lambda_{\delta}-\mathcal{H})^{-1}[\mathcal{P}[\phi^{\delta}]-(c_{\delta},c_{\delta})^T],\]
we know that $\phi^{\delta}$ converges to some $\phi^0\in X$ as $\delta\to0$. Note that $\lambda_0\ge\lambda_p>\max\{-b_1,-b_2\}$. Thus $\phi^0\in X^+$. Moreover, $\int_{l_1}^{l_2}[\phi^0_1(x)+\phi^0_2(x)]\dx=1$ and $\mathcal{L}[\phi^0]=\lambda_0\phi^0$. By the strong maximum principle (Lemma \ref{l2.1}), we have $\phi^0\in X^{++}$.

However, from Claim 1, $\mathcal{L}-\lambda\mathcal{I}$ is a bijection for any $\lambda>\lambda_p$. Hence $\lambda^0=\lambda_p$ and $\varphi_p=\phi^0$. Therefore, Step 1 is finished.

{\bf Step 2.} We prove that $\lambda_p>\max\{-b_1,-b_2\}$.

Since $\lambda_p\ge \max\{-b_1,-b_2\}$, we only need to show that $\lambda_p$ can not be equal to $\max\{-b_1,-b_2\}$. Otherwise, for any $\lambda>\lambda_p=\max\{-b_1,-b_2\}$, there exists a $\varphi_{\lambda}\in X^{++}$ such that
\bes\label{2.4}\lambda>\sup_{x\in[l_1,l_2],i=1,2}\frac{\mathcal{L}_i[\varphi_{\lambda}](x)}{(\varphi_{\lambda}(x))_i}.\ees
Without loss of generality, we assume that $\max\{-b_1,-b_2\}=-b_1\ge-b_2$. By virtue of {\bf (J)}, there exist small positive constants $\ep$ and $\sigma$ such that $J_{ij}(x)\ge\sigma$ if $|x|\le\ep$ with $i,j=1,2$. If $a_{11}>0$, using \eqref{2.4} we have
\bess
&&\int_{l_1}^{l_1+\ep}(\varphi_{\lambda})_1(x)\dx>\int_{l_1}^{l_1+\ep}\frac{a_{11}\int_{l_1}^{l_2}J_{11}(x-y)(\varphi_{\lambda})_1(y)\dy}{\lambda+b_1}\dx\\
&&\ge\int_{l_1}^{l_1+\ep}\int_{l_1}^{l_1+\ep}\frac{a_{11}J_{11}(x-y)(\varphi_{\lambda})_1(y)}{\lambda+b_1}\dy\dx
=\int_{l_1}^{l_1+\ep}(\varphi_{\lambda})_1(y)\big[\int_{l_1}^{l_1+\ep}\frac{a_{11}J_{11}(x-y)}{\lambda+b_1}\dx\big]\dy\\
&&\ge\frac{a_{11}\sigma}{\lambda+b_1}\int_{l_1}^{l_1+\ep}(\varphi_{\lambda})_1(y)\dy,
\eess
which implies $\frac{a_{11}\sigma}{\lambda+b_1}\le1$ for all $\lambda>-b_1$. Clearly, it is a contradiction.  If $a_{11}=0$, by $\mathcal{L}_1[\varphi]<\lambda(\varphi_{\lambda})_1$, we have
\[(\varphi_{\lambda})_1>\frac{a_{12}}{\lambda+b_1}\int_{l_1}^{l_2}J_{12}(x-y)(\varphi_{\lambda})_2(y)\dy {\rm ~ in ~ }[l_1,l_2],\]
which, combined with $\mathcal{L}_2[\varphi]<\lambda(\varphi_{\lambda})_2$, yields
\[(\varphi_{\lambda})_2>\frac{a_{12}a_{21}}{(\lambda+b_1)(\lambda+b_2)}\int_{l_1}^{l_2}\int_{l_1}^{l_2}J_{21}(x-z)J_{12}(z-y)(\varphi_{\lambda})_2(y)\dy {\rm d}z ~ {\rm ~ in ~ }[l_1,l_2].\]
Integrating the above inequality from $l_1$ to $l_1+\ep$ leads to
\bess
\int_{l_1}^{l_1+\ep}(\varphi_{\lambda})_2(x)\dx>\frac{a_{12}a_{21}}{(\lambda+b_1)(\lambda+b_2)}\int_{l_1}^{l_1+\ep}\int_{l_1}^{l_1+\ep}\int_{l_1}^{l_1+\ep}J_{21}(x-z)J_{12}(z-y)(\varphi_{\lambda})_2(y)\dy {\rm d}z\dx.
\eess
Thus we see
\[\int_{l_1}^{l_1+\ep}(\varphi_{\lambda})_2(x)\dx>\frac{a_{12}a_{21}\sigma^2\ep^2}{(\lambda+b_1)(\lambda+b_2)}\int_{l_1}^{l_1+\ep}(\varphi_{\lambda})_2(x)\dx,\]
which implies that
\[1>\frac{a_{12}a_{21}\sigma^2\ep^2}{(\lambda+b_1)(\lambda+b_2)} {\rm ~ for ~ all ~ }\lambda>-b_1\ge-b_2.\]
This is clearly a contradiction. So $\lambda_p>\max\{-b_1,-b_2\}$.
 Then Step 2 is complete. In a word, conclusion (1) is obtained.

(2) We now show that the algebraic multiplicity of $\lambda_p$ is one, i.e. dim $\bigcup_{k=1}^{\yy}N(\mathcal{L}-\lambda_p)^k=1$. In fact, we can prove a stronger conclusion, namely, $N(\mathcal{L}-\lambda_p)^k=N(\mathcal{L}-\lambda_p)$ for $k\ge2$ and dim $N(\mathcal{L}-\lambda_p)=1$. Let $\phi\in N(\mathcal{L}-\lambda_p)\setminus\{{\bf0}\}$. If one of the components of $\phi$ is negative somewhere on $[l_1,l_2]$, by Touching Lemma (Lemma \ref{l2.2}) with $(\mathcal{L}-\lambda_p\mathcal{I}, \varphi_p,\phi)$ in place of $(\mathcal{L},\varphi,\psi)$, we have $\phi$ is a negative constant multiple of $\varphi_p$. If $\phi\in X^+$, then similar to the above analysis, we can derive that $-\phi$ is a negative constant multiple of $\varphi_p$. Thus dim $N(\mathcal{L}-\lambda_p)=1$.

Let $\phi\in N(\mathcal{L}-\lambda_p)^2\setminus\{{\bf0}\}$. Then $(\mathcal{L}-\lambda_p)[\phi]=c\varphi_p$ for some $c\in\mathbb{R}$. If $c=0$, then $\phi\in N(\mathcal{L}-\lambda_p)$. If $c\neq0$, we define $\tilde{\phi}=-\phi/c$. Thus $(\mathcal{L}-\lambda_p)[\tilde{\phi}]=-\varphi_p$. If one of the components of $\phi$ is negative somewhere on $[l_1,l_2]$, using Touching Lemma again, we have $(\mathcal{L}-\lambda_p)[\tilde{\phi}]={\bf 0}$. If $\tilde{\phi}\in X^+$, from Strong Maximum Principle (Lemma \ref{l2.1}) we have $\tilde{\phi}\in X^{++}$. Then by virtue of Touching Lemma with $(\mathcal{L},\varphi,\psi)$ replaced by $(\mathcal{L}-\lambda_p,\tilde{\phi},-\varphi_p)$, we can obtain $(\mathcal{L}-\lambda_p)[\tilde{\phi}]={\bf 0}$. Hence $N(\mathcal{L}-\lambda_p)^2\subseteq N(\mathcal{L}-\lambda_p)$, which implies that $N(\mathcal{L}-\lambda_p)^{k+1}\subseteq N(\mathcal{L}-\lambda_p)^k$ for all $k\ge1$. Moreover, it is easy to see that $N(\mathcal{L}-\lambda_p)^{k+1}\supseteq N(\mathcal{L}-\lambda_p)^k$ for $k\ge1$. Therefore, $N(\mathcal{L}-\lambda_p)^{k}= N(\mathcal{L}-\lambda_p)$ for $k\ge1$, and further conclusion (2) is obtained.

(3) In view of Lemma \ref{l2.1} with $\mathcal{L}$ replaced by $\mathcal{L}-\lambda$, we have $\varphi\in X^{++}$. By Claim 1, we see $\lambda\le \lambda_p$. Thus $(\mathcal{L}-\lambda_p)[\varphi]\le{\bf 0}$. Then using Touching Lemma with $(\mathcal{L},\varphi,\psi)$ replaced by $(\mathcal{L}-\lambda_p,\varphi,-\varphi_p)$, we obtain $(\mathcal{L}-\lambda_p)[\varphi]={\bf 0}$. Thus conclusion (3) is derived.

(4) For convenience, we denote $\lambda_0=\sup_{\|\phi\|_2=1}\langle\mathcal{L}[\phi],\phi\rangle$.
Clearly, $\lambda_0$ is well defined. It suffices to show that $\lambda_0$ is an eigenvalue of $\mathcal{L}$ with a corresponding eigenfunction in $X^+\setminus\{{\bf0}\}$. To this end, we first prove $\lambda_0>\max\{-b_1,-b_2\}$. Without loss of generality, assume that $\max\{-b_1,-b_2\}=-b_1$. Let $((l_2-l_1)^{-1/2},0)$ be the testing function. Then we have
\bess
\lambda_0&=&\sup_{\|\phi\|_2=1}\langle\mathcal{L}[\phi],\phi\rangle\\
&\ge& a_{11}(l_2-l_1)^{-1/2}\int_{l_1}^{l_2}\int_{l_1}^{l_2}J_{11}(x-y)\dy\dx+a_{21}(l_2-l_1)^{-1/2}\int_{l_1}^{l_2}\int_{l_1}^{l_2}J_{21}(x-y)\dy\dx-b_1>b_1.
\eess
Thus $\lambda_0>\max\{-b_1,-b_2\}$ and matirx $\lambda_0I-H$ is invertible, which implies that $\lambda_0\mathcal{I}-\mathcal{H}$ has a bounded and linear inverse $(\lambda_0\mathcal{I}-\mathcal{H})^{-1}$.

By virtue of $a_{12}=a_{21}$, $J_{12}=J_{21}$ and the definition of $\lambda_0$, we see that $\langle\lambda_0\varphi-\mathcal{L}[\varphi],\psi\rangle$ is bilinear, symmetric and $\langle\lambda_0\varphi-\mathcal{L}[\varphi],\varphi\rangle\ge0$. So by Cauchy-Schwarz inequality, we have
 \bess
|\langle\lambda_0\varphi-\mathcal{L}[\varphi],\psi\rangle|
\le\langle\lambda_0\varphi-\mathcal{L}[\varphi],\varphi\rangle^{\frac12}\langle\lambda_0\psi-\mathcal{L}[\psi],\psi\rangle^{\frac12}\le \langle\lambda_0\varphi-\mathcal{L}[\varphi],\varphi\rangle^{\frac12}\|\lambda_0I-\mathcal{L}\|^{\frac{1}{2}}\|\psi\|_2,
\eess
which yields $\|\lambda_0\varphi-\mathcal{L}[\varphi]\|_2\le\langle\lambda_0\varphi-\mathcal{L}[\varphi],\varphi\rangle^{\frac12}\|\lambda_0I-\mathcal{L}\|^{\frac{1}{2}}$. Together with the definitions of $\lambda_0$ and $\mathcal{L}$, we derive that there exists a nonnegative sequence $\{\varphi^n\}$ with $\|\varphi^n\|_2=1$ such that
  \bes\label{2.5}
  \|\lambda_0\varphi^n-\mathcal{L}[\varphi^n]\|_2\to0 ~ ~{\rm  as ~}  n\to\yy.\ees

For convenience, let $\mathcal{T}[\varphi]=(\lambda_0\mathcal{I}-\mathcal{H})[\varphi]$. By Arzel{\`a}-Ascoli Theorem, $\mathcal{P}$ is compact and maps $E$ to $X$. Thus there exists a subsequence of $\{\varphi^n\}$, still denoted by itself, such that $\mathcal{P}[\varphi^n]\to \bar\varphi$ for some $\bar\varphi\in X$. Define $\mathcal{T}^{-1}[\bar\varphi]=\theta$ and $\theta\in X$. So $\dd\lim_{n\to\yy}\mathcal{T}^{-1}[\mathcal{P}[\varphi^n]]
=\mathcal{T}^{-1}[\bar\varphi]=\theta$ in $X$.  Notice that
 \[\mathcal{T}^{-1}[\mathcal{P}[\varphi^n]]-\varphi^n
=\mathcal{T}^{-1}[\mathcal{P}[\varphi^n]-\mathcal{T}[\varphi^n]]
 =\mathcal{T}[\mathcal{L}[\varphi^n]-\lambda_0\varphi^n].\]
By \eqref{2.5}, $\dd\lim_{n\to\yy}\varphi^n=\theta$ in $E$. Since $\varphi^n$ is nonnegative and $\theta\in X$, we have $\theta\in X^+$. Therefore, $\mathcal{T}^{-1}[\mathcal{P}[\theta]]=\theta$, i.e., $\mathcal{L}[\theta]=\lambda_0\theta$. Due to $\|\theta\|_2=1$, $\lambda_0$ is an eigenvalue of $\mathcal{L}$ with an corresponding eigenfunction $\theta\in X^+\setminus\{{\bf0}\} $. Then by conclusion (3), $\lambda_p=\lambda_0$. The proof is complete.
\end{proof}

Now we investigate the asymptotic behaviors of $\lambda_p$ on interval $[l_1,l_2]$ by using some elementray analysis instead of the variational characteristic.  Define
\[A=\begin{pmatrix}
  a_{11}-b_1 & a_{12} \\
   a_{21} & a_{22}-b_2
    \end{pmatrix}.\]
    Direct computations show there exists an eigenpair $(\lambda_A,\theta_A)$ with $\theta_A>0$ such that $(\lambda_AI-A)(\theta_A,1)^T=0$, where
    \bes\label{2.6}\lambda_A=\frac{a_{11}-b_1+a_{22}-b_2+\sqrt{(a_{11}-b_1-a_{22}+b_2)^2+4a_{12}a_{21}}}{2}, ~ ~ \theta_A=\frac{a_{12}}{\lambda_A-a_{11}+b_1}.\ees

 We need to introduce the following lemma since it is vital to our later arguments.
   \begin{lemma}\label{l2.3}Let $\lambda_p$ be the principal eigenvalue of \eqref{2.1}. Then the following statements are valid.
   \begin{enumerate}[$(1)$]
   \item If there exist $\phi=(\phi_1,\phi_2)^T\in X$ with $\phi_1,\phi_2\ge,\not\equiv0$ and $\lambda\in\mathbb{R}$ such that $\mathcal{L}[\phi]\le\lambda\phi$, then $\lambda_p\le\lambda$. Moreover, $\lambda_p=\lambda$ only if $\mathcal{L}[\phi]=\lambda\phi$.

   \item If there exist $\phi=(\phi_1,\phi_2)^T\in X^+\setminus\{{\bf0}\}$ and $\lambda\in\mathbb{R}$ such that $\mathcal{L}[\phi]\ge\lambda\phi$, then $\lambda_p\ge\lambda$. Moreover, $\lambda_p=\lambda$ only if $\mathcal{L}[\phi]=\lambda\phi$.
   \end{enumerate}
   \end{lemma}
   \begin{proof}
    (1) By Strong Maximum Principle (Lemma \ref{l2.1}) with $\mathcal{L}$ replaced by $\mathcal{L}-\lambda\mathcal{I}$, we have $\phi\in X^{++}$. Assume on the contrary that $\lambda_p>\lambda$. Then we see $(\mathcal{L}-\lambda)[M\varphi_p]>0$ for any $M>0$. Choose $M$ large enough such that $\phi-M\varphi_p$ is negative somewhere on $[l_1,l_2]$. Denote $\phi-M\varphi_p$ by $\psi$. Then we have $(\mathcal{L}-\lambda\mathcal{I})[\psi]<0$. Using Touching Lemma (Lemma \ref{l2.2}) with $(\mathcal{L},\varphi,\psi)$ replaced by $(\mathcal{L}-\lambda\mathcal{I},\phi,\psi)$, respectively, we derive that $\psi$ is a negative multiple of $\phi$, and $(\mathcal{L}-\lambda\mathcal{I})[\psi]=(\mathcal{L}-\lambda\mathcal{I})[\phi]={\bf0}$, which clearly implies $\mathcal{L}[\varphi_p]=\lambda\varphi_p$. By Proposition \ref{p2.1}, we have $\lambda=\lambda_p$. This contradicts $\lambda_p>\lambda$. Thus assertion (1) is obtained.

   (2) Arguing indirectly, we have $(\mathcal{L}-\lambda\mathcal{I})[\varphi_p]<{\bf0}$. In view of the assumption in (2), we see $(\mathcal{L}-\lambda\mathcal{I})[M\phi]\ge{\bf0}$ for any $M>0$. Let $M$ be sufficiently large such that $\tilde\psi:=\varphi_p-M\phi$ is negative somewhere on $[l_1,l_2]$. Moreover, since $(\mathcal{L}-\lambda\mathcal{I})[\tilde\psi]<{\bf0}$, by Touching Lemma with $(\mathcal{L}-\lambda\mathcal{I},\varphi_p,\tilde\psi)$ in place of $(\mathcal{L},\varphi,\psi)$, respectively, we have $\tilde{\psi}$ is a negative constant multiple of $\varphi_p$, and $(\mathcal{L}-\lambda\mathcal{I})[\tilde\psi]=(\mathcal{L}-\lambda\mathcal{I})[\varphi_p]={\bf0}$, which indicates $\lambda=\lambda_p$. This is a contradiction. So conclusion (2) is proved. The proof is finished.
     \end{proof}

Now we are in the position to show the asymptotic behaviors of $\lambda_p$ about interval $[l_1,l_2]$. Since it is clear that $\lambda_p$ relies only on the length $l_2-l_1$ of interval $[l_2,l_1]$, we consider $\lambda_p$ on $[-l,l]$ and rewrite it as $\lambda_p(l)$.

\begin{proposition}\label{p2.2} Let $\lambda_p(l)$ be the principal eigenvalue of \eqref{2.1}. Then the following results hold.
\begin{enumerate}[$(1)$]
\item $\lambda_p(l)$ is strictly increasing and continuous with respect to $l>0$.

\item $\lim_{l\to\yy}\lambda_{p}(l)=\lambda_A$,  where $\lambda_A$ is given by \eqref{2.6}.

\item $\lim_{l\to0}\lambda_p(l)=\max\{-b_1,-b_2\}$.
\end{enumerate}
\end{proposition}
\begin{proof}
{\rm (1)} Let $l_2>l_1>0$. Denote by $\varphi=(\varphi_1,\varphi_2)$ the positive eigenfunction of $\lambda_p(l_2)$. Simple computations yield
\bess\left\{\begin{aligned}
&a_{11}\int_{-l_1}^{l_1}J_{11}(x-y)\varphi_1(y)\dy+a_{12}\int_{-l_1}^{l_1}J_{12}(x-y)\varphi_2(y)\dy-b_{1}\varphi_1\le\lambda_p(l_2)\varphi_1, & &x\in[-l_1,l_1],\\[1mm]
&a_{21}\int_{-l_1}^{l_1}J_{21}(x-y)\varphi_1(y)\dy+a_{22}\int_{-l_1}^{l_1}J_{22}(x-y)\varphi_2(y)\dy-b_{2}\varphi_2\le\lambda_p(l_2)\varphi_2, & &x\in[-l_1,l_1],
 \end{aligned}\right.
 \eess
which holds strictly at $l_1$. By Lemma \ref{l2.3}, we have $\lambda_p(l_1)<\lambda_p(l_2)$. The monotonicity is obtained.

Next we prove $\lambda_p(l)$ is continuous for $l>0$. For any $l_0>0$ and $l_n\searrow l_0$, denote by $\varphi^0$ and $\varphi^n$ the positive eigenfunctions of $\lambda_p(l_0)$ and $\lambda_p(l_n)$, respectively. For clarity, we denote the operator $\mathcal{L}$ defined in $[-l_n,l_n]$ and $[-l_0,l_0]$ by $\mathcal{L}^n$ and $\mathcal{L}^0$, respectively. By monotonicity, $\lambda_p(l_n)$ is decreasing in $n$ and $\lambda_p(l_n)>\lambda_p(l_0)$. Thus $\lambda_{\yy}:=\lim_{n\to\yy}\lambda_p(l_n)\ge\lambda_p(l_0)$. If $\lambda_{\yy}>\lambda(l_0)$, then we have $(\mathcal{L}^0-\lambda_{\yy}\mathcal{I})[\varphi^0]<0$. By continuity, there exists $N>0$ such that $(\mathcal{L}^n-\lambda_{\yy}\mathcal{I})[\tilde{\varphi}^0]<0$ for all $n\ge N$, where $\tilde\varphi^0=\varphi^0$ for $|x|\le l_0$ and $\tilde{\varphi}^0=\varphi^0(l_0)$ for $|x|\in[l_0,l_n]$. Moveover, due to $\lambda_p(l_n)>\lambda_{\yy}$, we have $(\mathcal{L}^n-\lambda_{\yy}\mathcal{I})[\varphi^n]>0$. Define $\psi=\tilde\varphi^0-M\varphi^n$ with $M$ large enough such that $\psi$ is negative somewhere on $[-l_n,l_n]$. Clearly, $(\mathcal{L}^n-\lambda_{\yy}\mathcal{I})[\psi]<0$.
Then using Touching Lemma (Lemma \ref{l2.2}) with $(\mathcal{L},\varphi,\psi)$ replaced by $(\mathcal{L}^n-\lambda_{\yy}\mathcal{I},\tilde{\varphi}^0,\psi)$, respectively, we have $(\mathcal{L}^n-\lambda_{\yy}\mathcal{I})[\tilde{\varphi}^0]=0$ which implies $\lambda_{\yy}=\lambda_p(l_n)$ for $n\ge N$. This contradicts $\lambda_p(l_n)>\lambda_{\yy}$ for $n\ge1$. Thus $\lambda_{\yy}=\lambda_p(l_0)$.

For any $l_0>0$ and $l_n\nearrow l_0$, arguing as above we also can show $\lim_{n\to\yy}\lambda_p(l_n)=\lambda_p(l_0)$. Therefore, the continuity is obtained.

{\rm (2)} Recall that $\lambda_A$ and $\theta_A$ are given by \eqref{2.6}. Define $\bar{\varphi}=(\theta_A,1)^T$. Then we claim that $\mathcal{L}[\bar\varphi]\le \lambda_A\bar{\varphi}$ for all $l>0$ which, combined with Lemma \ref{l2.3}, yields
\bes\label{2.7}\lambda_p(l)\le\lambda_A.
\ees

Simple calculations leads to
\bess
&&a_{11}\int_{-l}^{l}J_{11}(x-y)\theta_A\dy+a_{12}\int_{-l}^{l}J_{12}(x-y)\dy-b_1\theta_A\le a_{11}\theta_A+a_{12}-b_1\theta_A=\lambda_A\theta_A,\\
&&a_{21}\int_{-l}^{l}J_{21}(x-y)\theta_A\dy+a_{22}\int_{-l}^{l}J_{22}(x-y)\dy-b_2\le a_{21}\theta_A+a_{22}-b_2=\lambda_A.
\eess
Thus our claim holds and \eqref{2.7} is obtained.

Define $\underline{\varphi}=(\underline{\varphi}_1(x),\underline\varphi_2(x))^T$ with
\[\underline{\varphi}_1(x)=\theta_A\xi(x), ~ ~ \underline{\varphi}_2(x)=\xi(x) ~ ~ {\rm and ~ ~ }\xi(x)=l-|x|.\]
From \cite[Lemma 5.2]{DN2}, there exists a $L_{\ep}>0$ such that when $l>L_{\ep}$, we have
\bes\label{2.8}
\int_{-l}^{l}J_{ij}(x-y)\xi(y)\dy\ge(1-\ep)\xi(x) ~ ~ {\rm for ~ } i,j=1,2, ~ x\in[-l,l].\ees
A straightforward calculation yields
\bess
&&a_{11}\int_{-l}^{l}J_{11}(x-y)\underline{\varphi}_1(y)\dy+a_{12}\int_{-l}^{l}J_{12}(x-y)\underline{\varphi}_2(y)\dy-b_1\underline{\varphi}_1\\
&&=a_{11}\int_{-l}^{l}J_{11}(x-y)\theta_A\xi(y)\dy+a_{12}\int_{-l}^{l}J_{12}(x-y)\xi(y)\dy-b_1\theta_A\xi\\
&&\ge a_{11}\theta_A(1-\ep)\xi+a_{12}(1-\ep)\xi-b_1\theta_A\xi\ge(\lambda_A-a_{11}\ep-\frac{a_{12}\ep}{\theta_A})\underline{\varphi}_1(x).
\eess
Similarly, we have
\[a_{21}\int_{-l}^{l}J_{21}(x-y)\underline\varphi_1(y)\dy+a_{22}\int_{-l}^{l}J_{22}(x-y)\underline\varphi_2(y)\dy-b_2\underline\varphi_2\ge(\lambda_A-a_{21}\theta_A\ep-a_{22}\ep)\underline\varphi_2(x).\]
By Lemma \ref{l2.3}, we obtain $\liminf_{l\to\yy}\lambda_p(l)\ge\lambda_A$, which, combined with \eqref{2.7}, completes the proof of assertion (2).

{\rm (3)} Clearly, $\lambda_p(l)>\max\{-b_1,-b_2\}$. Let $\underline{\psi}=(1,1)^T$. Direct computations show
\bess
&&a_{11}\int_{-l}^{l}J_{11}(x-y)\dy+a_{12}\int_{-l}^{l}J_{12}(x-y)\dy-b_1\le a_{11}\int_{-2l}^{2l}J_{11}(y)\dy+a_{12}\int_{-2l}^{2l}J_{12}(y)\dy-b_1\\
&&\le a_{11}\int_{-2l}^{2l}J_{11}(y)\dy+a_{12}\int_{-2l}^{2l}J_{12}(y)\dy+\max\{-b_1,-b_2\} ~ ~ {\rm in ~ }[-l,l];\\
&&a_{21}\int_{-l}^{l}J_{21}(x-y)\dy+a_{22}\int_{-l}^{l}J_{22}(x-y)\dy-b_2\\
&&\le a_{21}\int_{-2l}^{2l}J_{21}(y)\dy+a_{22}\int_{-2l}^{2l}J_{22}(y)\dy+\max\{-b_1,-b_2\}~ ~ {\rm in ~ }[-l,l].
\eess
By Lemma \ref{l2.3} again, we have
\[\lambda_p(l)\le\int_{-2l}^{2l}\bigg(\sum_{i,j=1}^{2}a_{ij}J_{ij}(y)\bigg)\dy+\max\{-b_1,-b_2\},\]
which implies $\limsup_{l\to0}\lambda_p(l)\le \max\{-b_1,-b_2\}$. Thus (3) is proved. The proof is finished.
\end{proof}

For later use, we now concretize principal eigenvalue problem \eqref{2.1} as follows.
\bes\left\{\begin{aligned}\label{2.9}
&d_1\int_{l_1}^{l_2}J_{11}(x-y)\varphi_1(y)\dy+c\int_{l_1}^{l_2}J_{12}(x-y)\varphi_2(y)\dy-d_1\varphi_1-a\varphi_1=\lambda\varphi_1, & &x\in[l_1,l_2],\\[1mm]
&G'(0)\int_{l_1}^{l_2}J_{21}(x-y)\varphi_1(y)\dy+d_2\int_{l_1}^{l_2}J_{22}(x-y)\varphi_2(y)\dy-d_{2}\varphi_2-b\varphi_2=\lambda\varphi_2, & &x\in[l_1,l_2].
 \end{aligned}\right.
 \ees
Then we investigate the asymptotic behaviors of $\lambda_p$ about diffusion coefficients $d_1$ and $d_2$. So we rewrite $\lambda_p$ as a binary function $\lambda_p(d_1,d_2)$ which, by Proposition \ref{p2.1}, is well defined on $[0,\yy)\times[0,\yy)$.

\begin{proposition}\label{p2.3} Let $\lambda_p(d_1,d_2)$ be given as above. Then the following statements are valid.
\begin{enumerate}[$(1)$]
  \item $\lambda_p(d_1,d_2)$ is continuous with respect to $(d_1,d_2)\in [0,\yy)\times[0,\yy)$.

  \item If $c=G'(0)$ and $J_{12}=J_{21}$, then $\lambda_p(d_1,d_2)$ is strictly decreasing in each variable $d_1$ and $d_2$.

  \item If we fix $d_2>0$, then $\lambda_p(d_1,d_2)\to\kappa_1$ as $d_1\to\yy$ where $\kappa_1$ is the principal eigenvalue of
  \[d_2\int_{l_1}^{l_2}J_{22}(x-y)\phi(y)\dy-d_2\phi-b\phi=\kappa\phi, ~ ~ x\in[l_1,l_2].\]
  If we fix $d_2=0$, then $\lambda_p(d_1,d_2)\to-b$ as $d_1\to\yy$.

  \item If we fix $d_1>0$, then $\lambda_p(d_1,d_2)\to\kappa_2$ as $d_2\to\yy$ where $\kappa_2$ is the principal eigenvalue of
  \[d_1\int_{l_1}^{l_2}J_{11}(x-y)\phi(y)\dy-d_1\phi-a\phi=\kappa\phi, ~ ~ x\in[l_1,l_2].\]
  If we fix $d_1=0$, $\lambda_p(d_1,d_2)\to-a$ as $d_2\to\yy$.

  \item $\lambda_p(d_1,d_2)\to-\yy$ as $(d_1,d_2)\to(\yy,\yy)$.
\end{enumerate}
\end{proposition}
\begin{proof}
{\rm (1)} Choose any $(\bar{d}_1,\bar{d}_2)$ and $(d_1,d_2)\in [0,\yy)\times[0,\yy)$. Denote by $(\varphi_1,\varphi_2)$ the corresponding positive eigenfunction of $\lambda_p(d_1,d_2)$. Direct computations yield
\bess
&&\bar{d}_1\int_{l_1}^{l_2}J_{11}(x-y)\varphi_1(y)\dy+c\int_{l_1}^{l_2}J_{12}(x-y)\varphi_2(y)\dy-\bar{d}_1\varphi_1-a\varphi_1\\
&&=\lambda_p(d_1,d_2)\varphi_1+(\bar{d}_1-d_1)\int_{l_1}^{l_2}J_{11}(x-y)\varphi_1(y)\dy-(\bar d_1-d_1)\varphi_1\\
&&\le \lambda_p(d_1,d_2)\varphi_1+2|\bar{d}_1-d_1|\max_{x\in[l_1,l_2]}\{\varphi_1(x)\}\frac{\varphi_1(x)}{\dd\min_{x\in[l_1,l_2]}\varphi_1(x)}.
\eess
Similarly, we have
\bess
&&G'(0)\int_{l_1}^{l_2}J_{21}(x-y)\varphi_1(y)\dy+\bar{d}_2\int_{l_1}^{l_2}J_{22}(x-y)\varphi_2(y)\dy-\bar{d}_2\varphi_2-b\varphi_2\\
&&\le\lambda_p(d_1,d_2)\varphi_2+2|\bar{d}_2-d_2|\max_{x\in[l_1,l_2]}\{\varphi_2(x)\}\frac{\varphi_2(x)}{\dd\min_{x\in[l_1,l_2]}\varphi_2(x)}.\eess
Let
\[K=\max\{\frac{\max_{x\in[0,l]}\{\varphi_1(x)\}}{\dd\min_{x\in[0,l]}\varphi_1(x)},\frac{\max_{x\in[0,l]}\{\varphi_2(x)\}}{\dd\min_{x\in[0,l]}\varphi_2(x)}\}.\]
Thus, using Lemma \ref{l2.3} we have
\bes\label{2.14}
\lambda_p(\bar{d}_1,\bar{d}_2)\le \lambda_p(d_1,d_2)+2K(|\bar{d}_1-d_1|+|\bar{d}_2-d_2|).
\ees
Analogously,
\[\lambda_p(\bar{d}_1,\bar{d}_2)\ge \lambda_p(d_1,d_2)-2K(|\bar{d}_1-d_1|+|\bar{d}_2-d_2|),\]
which together with \eqref{2.14} gives
\[|\lambda_p(\bar{d}_1,\bar{d}_2)-\lambda_p(d_1,d_2)|\le2K(|\bar{d}_1-d_1|+|\bar{d}_2-d_2|).\]
Then the continuity follows.

{\rm (2)} Since $c=G'(0)$ and $J_{12}=J_{21}$, by Proposition \ref{p2.1}, the variational characteristic holds. We only show the monotonicity of $\lambda_p(d_1,d_2)$ about $d_1$ since the other case is similar. Let us fix $d_2$ and choose any $0<\bar{d}_1<d_1$. Denote by $\varphi=(\varphi_1,\varphi_2)$ the corresponding positive eigenfunction of $\lambda_p(d_1,d_2)$ with $\|\varphi\|_2=1$. Then we see
\bess
\lambda_p(d_1,d_2)&=&d_1\int_{l_1}^{l_2}\int_{l_1}^{l_2}J_{11}(x-y)\varphi_1(y)\varphi_1(x)\dy\dx+c\int_{l_1}^{l_2}\int_{l_1}^{l_2}J_{12}(x-y)\varphi_2(y)\varphi_1(x)\dy\dx\\ &&-(d_1+a)\int_{l_1}^{l_2}\varphi^2_1(x)\dx+G'(0)\int_{l_1}^{l_2}\int_{l_1}^{l_2}J_{21}(x-y)\varphi_1(y)\varphi_2(x)\dy\dx\\
&&+d_2\int_{l_1}^{l_2}\int_{l_1}^{l_2}J_{22}(x-y)\varphi_2(y)\varphi_2(x)\dy\dx-(d_2+b)\int_{l_1}^{l_2}\varphi^2_2(x)\dx.
\eess
Following from \cite[Proposition 3.4]{CDLL}, we have
\bess
\int_{l_1}^{l_2}\int_{l_1}^{l_2}J_{11}(x-y)\varphi_1(y)\varphi_1(x)\dy\dx-\int_{l_1}^{l_2}\varphi^2_1(x)\dx<0.
\eess
Hence
\bess
\lambda_p(d_1,d_2)&<&\bar d_1\int_{l_1}^{l_2}\int_{l_1}^{l_2}J_{11}(x-y)\varphi_1(y)\varphi_1(x)\dy\dx+c\int_{l_1}^{l_2}\int_{l_1}^{l_2}J_{12}(x-y)\varphi_2(y)\varphi_1(x)\dy\dx\\ &&-(\bar d_1+a)\int_{l_1}^{l_2}\varphi^2_1(x)\dx+G'(0)\int_{l_1}^{l_2}\int_{l_1}^{l_2}J_{21}(x-y)\varphi_1(y)\varphi_2(x)\dy\dx\\
&&+d_2\int_{l_1}^{l_2}\int_{l_1}^{l_2}J_{22}(x-y)\varphi_2(y)\varphi_2(x)\dy\dx-(d_2+b)\int_{l_1}^{l_2}\varphi^2_2(x)\dx\\
&&\le \lambda_p(\bar{d}_1,d_2).
\eess
The monotonicity is obtained.

(3) Since conclusions (3) and (4) are parallel, we only prove conclusion (3). Recalling \eqref{2.8}, we know $\lambda_p(d_1,d_2)\le\lambda_A$. Let $\underline{\varphi}=(0,\phi)$ where $\phi$ is the corresponding positive eigenfunction of $\kappa_1$. Clearly, $\mathcal{L}[\underline{\varphi}]\ge\kappa_1\underline{\varphi}$. By Lemma \ref{l2.1}, $\lambda_p(d_1,d_2)\ge\kappa_1$.

We first consider the case $d_2>0$.
Clearly, it suffices to prove that for any sequence $\{d^n_1\}$ with $d^n_1\to\yy$ as $n\to\yy$, there exists a subsequence, still denoted by itself, satisfying $\lambda_p(d^n_1,d_2)\to\kappa_1$ as $n\to\yy$. For convenience, denote $\lambda_p(d^n_1,d_2)$ by $\lambda^n_p$ since we fix $d_2$. Let $\varphi^n=(\varphi^n_1,\varphi^n_2)$ be the corresponding positive eigenfunction of $\lambda^n_p$ with $\|\varphi^n\|_{X}=1$. Notice that $\lambda^n_p$ is bounded independent of $n$. Thus there exists a subsequence of $\{n\}$, still denoted by itself, such that $(\varphi^n_1,\varphi^n_2)$ converges weakly to $(\psi_1,\psi_2)$ with $\psi_i\in L^2([l_1,l_2])$, and $\lambda^n_p\to\lambda_{\yy}\ge\kappa_1$ as $n\to\yy$. Due to $\varphi^n\in X^{++}$, we have $\psi_i\ge0$ for $i=1,2$. We claim that $\psi_1\equiv0$.
Dividing $\mathcal{L}_1[\varphi^n]=\lambda_p^n\varphi^n_1$ by $d^n_1$ and then letting $n\to\yy$ yield
\[\int_{l_1}^{l_2}J_{11}(x-y)\varphi^n_1(y)\dy-\varphi^n_1\to0 ~ ~ {\rm uniformly ~ in ~ }[l_1,l_2].\]
Since $\varphi^n_1$ converges weakly to $\psi_1$ and operator $\int_{l_1}^{l_2}J_{11}(x-y)\varphi^n_1(y)\dy$ is compact, we have as $n\to\yy$,
\[\int_{l_1}^{l_2}J_{11}(x-y)\varphi^n_1(y)\dy\to\int_{l_1}^{l_2}J_{11}(x-y)\psi_1(y)\dy ~ ~ {\rm uniformly ~ in ~ }[l_1,l_2].\]
Hence
\[\varphi^n_1\to \int_{l_1}^{l_2}J_{11}(x-y)\psi_1(y)\dy~ ~ {\rm uniformly ~ in ~ }[l_1,l_2].\]
By the uniqueness of weak limit, we have
\[\psi_1(x)=\int_{l_1}^{l_2}J_{11}(x-y)\psi_1(y)\dy ~ ~ {\rm for ~ }x\in[l_1,l_2].\]
If there exists some $x_0\in[l_1,l_2]$ such that $\psi_1(x_0)>0$, then it is not hard to show that $\psi_1(x)>0$ in $[l_1,l_2]$, which implies that $(0,\psi_1)$ is the principal eigenpair of problem
\[\int_{l_1}^{l_2}J_{11}(x-y)\omega(y)\dy-\omega(x)=\xi\omega.\]
However, from \cite[Proposition 3.4]{CDLL}, the principal eigenvalue $\xi$ must be less than $0$. This contradiction implies $\psi_1\equiv0$, and thus our claim holds.

Since $\varphi^n_1\to0$ in $C([l_1,l_2])$ and $\|\varphi^n\|_{X}=1$, we have $\|\varphi^n_2\|_{C([l_1,l_2])}\to1$ as $n\to\yy$. Note that $\varphi^n_2\to\psi_2$ weakly in $L^2([l_1,l_2])$ and operator $\int_{l_1}^{l_2}J_{22}(x-y)\varphi^n_2(y)\dy$ is compact, we have
\bess\int_{l_1}^{l_2}J_{22}(x-y)\varphi^n_2(y)\dy\to\int_{l_1}^{l_2}J_{22}(x-y)\psi_2(y)\dy ~ ~ {\rm uniformly~ in ~ }[l_1,l_2].
\eess
Moreover, due to $\varphi^n_1\to0$ in $C([l_1,l_2])$ as $n\to\yy$, as $n\to\yy$ we have
\[d_2\int_{l_1}^{l_2}J_{22}(x-y)\varphi^n_2(y)\dy-d_2\varphi^n_2-b\varphi^n_2-\lambda^n_p\varphi^n_2\to0 ~ ~ {\rm uniformly ~ in}~ [l_1,l_2].\]
By \cite[Proposition 3.4]{CDLL}, we see $d_2+b+\lambda^n_p\ge d_2+b+\kappa_1>0$.
So as $n\to\yy$, we have
\bess\varphi^n_2(x)\to\frac{\dd d_2\int_{l_1}^{l_2}J_{22}(x-y)\psi_2(y)\dy}{d_2+b+\kappa_1} ~ ~ {\rm uniformly ~ in}~ [l_1,l_2].\eess
By the uniqueness of weak limit, we obtain
\bes\label{2.11}d_2\int_{l_1}^{l_2}J_{22}(x-y)\psi_2(y)\dy-d_2\psi_2-b\psi_2=\lambda_{\yy}\psi_2.\ees
 Recall that $\|\varphi^n_2\|_{C([l_1,l_2])}\to1$ as $n\to\yy$. We obtain $\|\psi_2\|_{C([l_1,l_2])}=1$. Together with \eqref{2.11}, we easily derive that $\psi_2>0$ in $[l_1,l_2]$, which indicates that $\lambda_{\yy}=\kappa_1$.

 If we fix $d_2=0$, then we still have $\varphi^n_1\to0$ in $C([l_1,l_2])$ and thus $\|\varphi^n_2\|_{C([l_1,l_2])}\to1$ as $n\to\yy$. By $\mathcal{L}_2[\varphi^n]=\lambda^n_p\varphi^n_2$, we derive $|\lambda^n_p+b|\|\varphi^n_2\|_{C([l_1,l_2])}\to0$ as $n\to\yy$ which, combined with $\|\varphi^n_2\|_{C([l_1,l_2])}\to1$, leads to $\lambda_{\yy}=-b$.
Thus conclusion (4) is proved.

(5) It can be seen from \cite[Proposition 3.4]{CDLL} that for $i=1,2$, the following eigenvalue problem
\[\int_{l_1}^{l_2}J_{ii}(x-y)\omega(y)\dy-\omega(x)=\lambda\omega(x) ~ ~ {\rm for ~ }x\in[l_1,l_2]\]
has an eigenvalue $\lambda_i$ with a positive eigenfunction $\omega_i$ satisfying $\|\omega_i\|_{C([l_1,l_2])}=1$. Moreover, $\lambda_i\in(-1,0)$. Let
\bess
&\underline{d}=\min\{d_1,d_2\}, ~ ~ \underline{\lambda}=\min\{\lambda_1,\lambda_2\}\\ &K=a+b+\frac{c}{\min_{x\in[l_1,l_2]}\omega_1(x)}+\frac{G'(0)}{\min_{x\in[l_1,l_2]}\omega_2(x)}, ~ ~ \bar{\omega}=(\omega_1,\omega_2)^T.
\eess
Next we show $\mathcal{L}[\bar{\omega}]\le(\underbar{d}\underline{\lambda}+K)\bar{\omega}$. Once it is done, by Lemma \ref{l2.3}, we get $\lambda_p(d_1,d_2)\le\underline{d}\underline{\lambda}+K$. Then conclusion (6) obviously follows from the fact that $\underline{\lambda}<0$ and $K$ is independent of $(d_1,d_2)$.

Simple computations yield
\bess
&&d_1\int_{l_1}^{l_2}J_{11}(x-y)\omega_1(y)\dy+c\int_{l_1}^{l_2}J_{12}(x-y)\omega_2(y)\dy-d_1\omega_1-a\omega_1\\
&&\le (d_1\lambda_1+a+\frac{c}{\min_{x\in[l_1,l_2]}\omega_1(x)})\omega_1\le (d_1\lambda_1+K)\omega_1
\eess
Similarly we have $\mathcal{L}_2[\bar{\omega}]\le(d_2\lambda_2+K)\omega_2$.
Therefore, $\mathcal{L}[\bar{\omega}]\le(\underbar{d}\underline{\lambda}+K)\bar{\omega}$ and (6) is proved. The proof is ended.
\end{proof}

With the aid of the above results, we now discuss a fixed boundary problem which is associated to \eqref{1.7}. Let us start by considering the following steady state problem
\bes\left\{\begin{aligned}\label{2.12}
&d_1\int_{l_1}^{l_2}J_{11}(x-y)U(y)\dy-d_1U-aU+c\int_{l_1}^{l_2}J_{12}(x-y)V(y)\dy=0, & & x\in[l_1,l_2]\\
&d_2\int_{l_1}^{l_2}J_{22}(x-y)V(y)\dy-d_2V-bV+G(\int_{l_1}^{l_2}J_{21}(x-y)U(y)\dy)=0, & & x\in[l_1,l_2].
 \end{aligned}\right.
 \ees

 To emphasize the dependence of $\lambda_p$ on interval $[l_1,l_2]$, we denote it by $\lambda_p(l_1,l_2)$. Since the following result can be proved by using similar methods in \cite[Theorem 3.10 and Lemma 3.11]{WD1} or \cite[Propositions 2.10 and 2.11]{NV}, we omit the details here.

 \begin{lemma}\label{l2.4}Let $\lambda_p(l_1,l_2)$ be defined as above. Then the following statements are valid.
 \begin{enumerate}[$(1)$]
 \item If $\lambda_p(l_1,l_2)>0$, then problem \eqref{2.12} has a unique solution $(U,V)\in X^{++}$ and $(u^*-U,v^*-V)\in X^{++}$. Denote $(U,V)$ by $(U_{l_1,l_2},V_{l_1,l_2})$. Then  $(U_{l_1,l_2},V_{l_1,l_2})\ge (U_{\tilde l_1,\tilde l_2},V_{\tilde l_1,\tilde l_2})$ for all large $[\tilde{l}_1,\tilde{l}_2]\subseteq[l_1,l_2]$ and $(U_{l_1,l_2},V_{l_1,l_2})\to(u^*,v^*)$ in locally uniformly in $\mathbb{R}$ as $-l_1,l_2\to\yy$.

 \item If $\lambda_p(l_1,l_2)\le0$, then ${\bf 0}$ is the unique nonnegative solution of \eqref{2.12}.
 \end{enumerate}
 \end{lemma}

Next we consider the following fixed boundary problem.

\bes\left\{\begin{aligned}\label{2.21}
&u_t=d_1\int_{l_1}^{l_2}J_{11}(x-y)u(y)\dy-d_1u-au+c\int_{l_1}^{l_2}J_{12}(x-y)v(y)\dy, & & t>0, ~ ~ x\in[l_1,l_2]\\
&v_t=d_2\int_{l_1}^{l_2}J_{22}(x-y)v(y)\dy-d_2v-bv+G(\int_{l_1}^{l_2}J_{21}(x-y)u(y)\dy), & & t>0, ~ ~ x\in[l_1,l_2],\\
&u(0,x)=\tilde{u}_0(x), ~ ~v(0,x)=\tilde{v}_0(x),
 \end{aligned}\right.
 \ees
 where $(\tilde u_0,\tilde v_0)\in X^+\setminus\{{\bf 0}\}$. It is worthy mentioning that in this paper we will often use comparison principles, including the comparison principles for fixed boundary problem and for free boundary problem. Since these comparison principles are well known to us, we do not prove them here. The readers can refer to \cite[Lemmas 2.1 and 2.3]{WD1} for some similar versions.

 \begin{lemma}\label{l2.5}Let $(u,v)$ be the unique solution of \eqref{2.21}. Then the following statements are valid.
 \begin{enumerate}[$(1)$]
 \item If $\lambda_p(l_1,l_2)>0$, then $(u(t,x),v(t,x))\to(U,V)$ in $X$ as $t\to\yy$.

 \item If $\lambda_p(l_1,l_2)\le0$, then $(u(t,x),v(t,x))\to{\bf0}$ in $X$ as $t\to\yy$. Moreover, if $\lambda_p(l_1,l_2)<0$, then $(e^{kt}u(t,x),e^{kt}v(t,x))\to{\bf0}$ in $X$ for all $k\in(0,-\lambda_p(l_1,l_2))$ as $t\to\yy$.
 \end{enumerate}
 \end{lemma}
 \begin{proof}The convergence results in $X$ can be proved by using similar methods as in  \cite[Proposition 3.4]{DN4} or \cite[Proposition 2.12]{NV}. Hence we only give the sketch for proof of the exponential stability.  Let $\varphi=(\varphi_1,\varphi_2)^T$ be the corresponding positive function of $\lambda_p(l_1,l_2)$. Define $\bar{u}=Me^{-kt}\varphi_1(x)$ and $\bar{v}=Me^{-kt}\varphi_2(x)$ with positive constants $M$ and $k$ to be determined later. It is not hard to verify that there exist suitable $M$ and $k$ such that $(\bar{u},\bar{v})$ is an upper solution of \eqref{2.21}. Then the desired result follows from a comparison argument. The details are ignored here.
 \end{proof}

\section{Dynamics of \eqref{1.7}}
In this section, we investigate the dynamics of \eqref{1.7}. We first show spreading-vanishing dichotomy holds, and then discuss the criteria determining when spreading or vanishing happens. Note that the well-posedeness of \eqref{1.7}, Theorem \ref{t1.1}, can be proved by using analogous methods as in the proof of \cite[Theorem 1.1]{WD1}. The details are thus omitted here. From Theorem \ref{t1.1}, we see that $h(t)$ is strictly increasing in $[0,\yy)$,  and thus $g_{\yy}:=\lim_{t\to\yy}g(t)\in(h_0,\yy]$ and $h_{\yy}:=\lim_{t\to\yy}h(t)\in(h_0,\yy]$ are well defined. We call the case $h_{\yy}-g_{\yy}<\yy$ {\it vanishing}, and call $h_{\yy}-g_{\yy}=\yy$ {\it spreading}.

\begin{lemma}\label{l3.1}If $h_{\yy}-g_{\yy}<\yy$, then $\lambda_p(g_{\yy},h_{\yy})\le0$ and $\lim_{t\to\yy}\|u(t,x)+v(t,x)\|_{C([g(t),h(t)])}=0$. Moreover, if $\lambda_p(g_{\yy},h_{\yy})<0$, then $\lim_{t\to\yy}e^{kt}\|u(t,x)+v(t,x)\|_{C([g(t),h(t)])}=0$ for all $k<-\lambda_p(g_{\yy},h_{\yy})$.
\end{lemma}
\begin{proof}We first prove that if $h_{\yy}-g_{\yy}<\yy$, then $\lambda_p(g_{\yy},h_{\yy})\le0$. Otherwise, by the continuity of $\lambda_p(l_1,l_2)$ on $[l_1,l_2]$, there exist small $\ep>0$ and $\delta>0$ such that $\lambda_p(g_{\yy}+\ep,h_{\yy}-\ep)>0$ and $\min\{J_{11}(x),J_{22}(x)\}\ge\delta$ for $|x|\le 2\ep$ due to {\bf (J)}.  Moreover, there is $T>0$ such that $g(t)<g_{\yy}+\ep$ and $h(t)>h_{\yy}-\ep$ for $t\ge T$. Hence the solution component $(u,v)$ of \eqref{1.7} satisfies that for all $t>T$ and $x\in[g_{\yy}+\ep,h_{\yy}-\ep]$,
\bess\left\{\begin{aligned}
&u_t\ge d_1\int_{g_{\yy}+\ep}^{h_{\yy}-\ep}J_{11}(x-y)u(y)\dy-d_1u-au+c\int_{g_{\yy}+\ep}^{h_{\yy}-\ep}J_{12}(x-y)v(t,y)\dy,\\
&v_t\ge d_2\int_{g_{\yy}+\ep}^{h_{\yy}-\ep}J_{22}(x-y)v(y)\dy-d_2v-bv+G(\int_{g_{\yy}+\ep}^{h_{\yy}-\ep}J_{21}(x-y)u(y)\dy), \\
&u(T,x)>0, ~ ~v(T,x)>0.
 \end{aligned}\right.
 \eess
 Let $(\underline{u},\underline{v})$ be the solution of \eqref{2.21} with $[l_1,l_2]=[g_{\yy}+\ep,h_{\yy}-\ep]$, $\tilde{u}_0(x)=u(T,x)$ and $\tilde{v}_0(x)=v(T,x)$. Noticing that $\lambda_p(g_{\yy}+\ep,h_{\yy}-\ep)>0$, by Lemma \ref{l2.5} we have $(\underline{u},\underline{v})\to(U,V)$ in $X$ as $t\to\yy$ where $(U,V)$ is the positive steady state of \eqref{2.12} with $[l_1,l_2]=[g_{\yy}+\ep,h_{\yy}-\ep]$. Furthermore, by comparison principle, we see $u(t+T,x)\ge \underline{u}(t,x)$ and $v(t+T,x)\ge \underline{v}(t,x)$ for $x\in[g_{\yy}+\ep,h_{\yy}-\ep]$, which implies that $\liminf_{t\to\yy}(u(t,x),v(t,x))\ge(U,V)$ uniformly in $[g_{\yy}+\ep,h_{\yy}-\ep]$. Thus there exist a small $\sigma>0$ and a large $T_1\gg T$ such that $u(t,x)\ge\sigma$ and $v(t,x)\ge\sigma$ for $t\ge T_1$ and $[g_{\yy}+\ep,h_{\yy}-\ep]$. In view of the equation of $h(t)$, we have for $t>T_1$
 \bess
 h'(t)&=&\mu_1\int_{g(t)}^{h(t)}\int_{h(t)}^{\yy}J_{11}(x-y)u(t,x)\dy\dx+\mu_2\int_{g(t)}^{h(t)}\int_{h(t)}^{\yy}J_{22}(x-y)v(t,x)\dy\dx\\
 &\ge&\mu_1\int_{h_{\yy}-\frac{3\ep}{2}}^{h_{\yy}-\ep}\int_{h_{\yy}}^{h_{\yy}+\frac{\ep}{2}}J_{11}(x-y)\sigma\dy\dx
 +\mu_2\int_{h_{\yy}-\frac{3\ep}{2}}^{h_{\yy}-\ep}\int_{h_{\yy}}^{h_{\yy}+\frac{\ep}{2}}J_{22}(x-y)\sigma\dy\dx\\
 &\ge&(\mu_1+\mu_2)\delta\sigma>0,
 \eess
 which clearly contradicts $h_{\yy}<\yy$. Thus $\lambda_p(g_{\yy},h_{\yy})\le0$.

 Let $(\bar{u},\bar{v})$ be the solution of \eqref{2.21} with $[l_1,l_2]=[g_{\yy},h_{\yy}]$, $\tilde{u}_0(x)=\|u_0\|_{C([-h_0,h_0])}$ and $\tilde{v}_0(x)=\|v_0\|_{C([-h_0,h_0])}$. Clearly, $\bar{u}(t,x)\ge u(t,x)$ and $\bar{v}\ge v(t,x)$ for $t\ge0$ and $x\in[g(t),h(t)]$. Note that $\lambda_p(g_{\yy},h_{\yy})\le0$. Then the convergence results follow from Lemma \ref{l2.5}. The proof is ended.
\end{proof}

The above result shows that if vanishing happens, then the epidemic will disappear in the long run. However, the following result implies that if spreading occurs, then the epidemic will be successfully transmitted to whole space. Since one can prove it by similar lines in the proof of \cite[Lemmas 4.5 and 4.6]{WD1}, we omit the details.

\begin{lemma}\label{l3.2}If $h_{\yy}-g_{\yy}=\yy$( necessarily $\mathcal{R}_0>1$, see Lemma \ref{l3.3} ), then $-g_{\yy}=\yy$, $h_{\yy}=\yy$ and $(u(t,x),v(t,x))\to(u^*,v^*)$ in $C_{loc}(\mathbb{R})$ as $t\to\yy$.
\end{lemma}

Clearly, the spreading-vanishing dichotomy, Theorem \ref{t1.2}, follows from Lemmas \ref{l3.1} and \ref{l3.2}. Next we discuss when spreading or vanishing happens.

\begin{lemma}\label{l3.3} If $\mathcal{R}_0\le1$, then vanishing happens. Particularly,
\[h_{\yy}-g_{\yy}\le\int_{-h_0}^{h_0}\left[u_0(x)+\frac{c}{b}v_0(x)\right]\dx+\min\kk\{\frac{d_1}{\mu_1},\,\frac{cd_2}{b\mu_2}\rr\}2h_0.\]
\end{lemma}
\begin{proof}
Remember $\mathcal{R}_0\le1$. By a series of simple computations, we have
\bess
&&\frac{d}{dt}\int_{g(t)}^{h(t)}\left[u(t,x)+\frac{c}{b}v(t,x)\right]\dx= \int_{g(t)}^{h(t)}\left[u_t(t,x)+\frac{c}{b}v_t(t,x)\right]\dx\\
&&=\int_{g(t)}^{h(t)}\bigg\{d_1\int_{g(t)}^{h(t)}J_{11}(x-y)u(t,y)\dy-d_1u-au+c\int_{g(t)}^{h(t)}J_{12}(x-y)v(t,y)\dy\\
&&+\frac{c}{b}\bigg[d_2\int_{g(t)}^{h(t)}J_{22}(x-y)v(t,y)\dy-d_2v-bv+G(\int_{g(t)}^{h(t)}J_{21}(x-y)u(t,y)\dy)\bigg]\bigg\}\dx\\
&&=-\int_{g(t)}^{h(t)}\int_{h(t)}^{\yy}\bigg[d_1J_{11}(x-y)u(t,x)+\frac{cd_2}{b}J_{22}(x-y)v(t,x)\bigg]\dy\dx\\
&&-\int_{g(t)}^{h(t)}\int_{-\yy}^{g(t)}\bigg[d_1J_{11}(x-y)u(t,x)+\frac{cd_2}{b}J_{22}(x-y)v(t,x)\bigg]\dy\dx\\
&&+\int_{g(t)}^{h(t)}\bigg[-au(t,x)+c\int_{g(t)}^{h(t)}J_{12}(x-y)v(t,y)\dy-cv(t,x)+G(\int_{g(t)}^{h(t)}J_{21}(x-y)u(t,y)\dy)\bigg]\dx\\
&&\le-\min\kk\{\frac{d_1}{\mu_1},\,\frac{cd_2}{b\mu_2}\rr\}(h'(t)-g'(t)).
\eess
Hence we derive
\[\frac{d}{dt}\int_{g(t)}^{h(t)}\left[u(t,x)+\frac{c}{b}v(t,x)\right]\dx\le-\min\kk\{\frac{d_1}{\mu_1},\,\frac{cd_2}{b\mu_2}\rr\}(h'(t)-g'(t)).\]
Integrating the above inequality from 0 to t yields
\[h(t)-g(t)\le\int_{-h_0}^{h_0}\left[u_0(x)+\frac{c}{b}v_0(x)\right]\dx+\min\kk\{\frac{d_1}{\mu_1},\,\frac{cd_2}{b\mu_2}\rr\}2h_0,\]
which completes the proof.
\end{proof}

Now we deal with the case $\mathcal{R}_0>1$. For the principal eigenvalue problem \eqref{2.9}, a straightforward calculation gives
\[\lambda_A=\frac{-a-b+\sqrt{(a-b)^2+4cG'(0)}}{2}.\]
Clearly, $\mathcal{R}_0>1$ is equivalent to $\lambda_A>0$. By Proposition \ref{p2.2}, there exists a unique critical length $L^*$, depending only on $(d_i, a, b, c, G'(0), J_{ij})$ for $i,j=1,2$, such that $\lambda_p(l_1,l_2)=0$ if $l_2-l_1=L^*$, and $\lambda_p(l_1,l_2)(l_2-l_1-L^*)>0$ if $l_2-l_1\neq L^*$. Then from Lemma \ref{l3.1}, it follows that if $2h_0\ge L^*$, then spreading happens. So we have the following result.

\begin{lemma}\label{l3.4}Let $L^*$ be defined as above. Then if vanishing occurs, then $h_{\yy}-g_{\yy}\le L^*$, which implies that spreading happens when $2h_0\ge L^*$.
\end{lemma}

\begin{lemma}\label{l3.5}If $2h_0<L^*$, there exists a $\bar\mu>0$ such that vanishing happens if $\mu_1+\mu_2\le \bar\mu$.
\end{lemma}
\begin{proof}
Due to $2h_0<L^*$, by Proposition \ref{p2.2} we know there exists a small $\ep>0$ such that $\lambda_p(-h_0(1+\ep),h_0(1+\ep))<0$. For convenience, denote $h_1=h_0(1+\ep)$. Let $\varphi=(\varphi_1,\varphi_2)$ be the positive eigenfunction of $\lambda_p(-h_1,h_1)$ with $\|\varphi\|_{X}=1$. Define
\bess
\bar{h}(t)=h_0\left[1+\ep(1-e^{-\delta t})\right], \bar{g}(t)=-\bar{h}(t), ~ ~ \bar{u}(t,x)=Me^{-\delta t}\varphi_1, ~ ~ \bar{v}=Me^{-\delta t}\varphi_2
\eess
with $0<\delta\le -\lambda_p(-h_1,h_1)$ and $M$ large enough such that $M\varphi_1(x)\ge u_0(x)$ and $M\varphi_2(x)\ge v_0(x)$ for $x\in[-h_1,h_1]$.
Direct calculations yield that for $t>0$ and $x\in[-\bar{h}(t),\bar{h}(t)]$
 \bess
 &&\bar{u}_t-d_1\int_{-\bar{h}}^{\bar{h}}J_{11}(x-y)\bar{u}(t,y)\dy+d_1\bar{u}+a\bar{u}-c\int_{-\bar{h}}^{\bar h}J_{12}(x-y)\bar{v}(t,y)\dy\\
 &&\ge Me^{-\delta t}\left[-\delta\varphi_1-d_1\int_{-h_1}^{h_1}J_{11}(x-y)\varphi_1(y)\dy+d_1\varphi_1+a\varphi_1-c\int_{-h_1}^{h_1}J_{12}(x-y)\varphi_2(y)\dy\right]\\
 &&\ge Me^{-\delta t}\left(-\delta \varphi_1-\lambda_p(-h_1,h_1)\varphi_1\right)\ge0.
\eess
Using the assumption on $G$, we can derive
\bess
&&\bar{v}_t-d_2\int_{-\bar{h}}^{\bar{h}}J_{22}(x-y)\bar{v}(t,y)\dy+d_2\bar{v}+b\bar{v}-G(\int_{-\bar{h}}^{\bar{h}}J_{21}(x-y)\bar{u}(t,y)\dy)\\
&&\ge Me^{-\delta t}\left[-\delta\varphi_2-d_2\int_{-h_1}^{h_1}J_{22}(x-y)\varphi_2(y)\dy+d_2\varphi_2+b\varphi_2-\frac{G(\dd\int_{-h_1}^{h_1}J_{21}(x-y)\bar{u}(t,y)\dy)}{Me^{-\delta t}}\right]\\
&&\ge Me^{-\delta t}\left[-\delta\varphi_2-d_2\int_{-h_1}^{h_1}J_{22}(x-y)\varphi_2(y)\dy+d_2\varphi_2+b\varphi_2-G'(0)\int_{-h_1}^{h_1}J_{21}(x-y)\varphi_1(y)\dy\right]\\
&&=Me^{-\delta t}(-\delta\varphi_2-\lambda_p(-h_1,h_1)\varphi_2)\ge0.
\eess
Moreover,
\bess
&&\mu_1\int_{-\bar{h}(t)}^{\bar{h}(t)}\int_{\bar{h}(t)}^{\yy}J_{11}(x-y)\bar{u}(t,x)\dy\dx+\mu_2\int_{-\bar{h}(t)}^{\bar{h}(t)}\int_{\bar{h}(t)}^{\yy}J_{22}(x-y)\bar{v}(t,x)\dy\dx\\
&&=\mu_1 Me^{-\delta t}\int_{-\bar{h}(t)}^{\bar{h}(t)}\int_{\bar{h}(t)}^{\yy}J_{11}(x-y)\varphi_1(x)\dy\dx+\mu_2 Me^{-\delta t}\int_{-\bar{h}(t)}^{\bar{h}(t)}\int_{\bar{h}(t)}^{\yy}J_{22}(x-y)\varphi_2(x)\dy\dx\\
&&\le(\mu_1+\mu_2)Me^{-\delta t}2h_1\le\ep\delta h_0e^{-\delta t}=\bar{h}'(t)
\eess
provided that $\mu_1+\mu_2\le\frac{\ep\delta h_0}{2Mh_1}$.  Therefore, by comparison principle we have $\bar{u}(t,x)\ge u(t,x)$, $\bar{v}(t,x)\ge v(t,x)$, $-\bar{h}(t)\le g(t)$ and $\bar{h}(t)\ge h(t)$ for $t\ge0$ and $x\in[g(t),h(t)]$, which implies vanishing will happen. The proof is ended.
\end{proof}

\begin{remark}\label{r3.1}From the above proof, we can see that $M\to0$ as $\|u_0(x)+v_0(x)\|_{C([-h_0,h_0])}\to0$. Hence if $\|u_0(x)+v_0(x)\|_{C([-h_0,h_0])}$ is small enough, then vanishing will happen.
\end{remark}

Now we parameterize the initial function $(u_0,v_0)$ as $\tau(\vartheta_1,\vartheta_2)$ with $\tau>0$ and $(\vartheta_1,\vartheta_2)$ satisfying {\bf(I)}. The next result implies that if $\tau$ is sufficiently large, then spreading will occur.

\begin{lemma}\label{l3.6}Assume that $2h_0<L^*$ and $J_{ii}(x)>0$ in $\mathbb{R}$ for $i=1,2$. Then spreading happens if $\tau$ is sufficiently large.
\end{lemma}
\begin{proof}By way of contradiction, we suppose that vanishing happens for all $\tau>0$. By Lemma \ref{l3.1}, we have $h_{\yy}-g_{\yy}\le L^*$. In view of the equation of $h$, we have
\bess
&&h'(t)=\mu_1\int_{g(t)}^{h(t)}\int_{h(t)}^{\yy}J_{11}(x-y)u(t,x)\dy\dx+\mu_2\int_{g(t)}^{h(t)}\int_{h(t)}^{\yy}J_{22}(x-y)v(t,x)\dy\dx\\
&&\ge\min\{\mu_1\int_{L^*}^{\yy}J_{11}(y)\dy, \mu_2\int_{L^*}^{\yy}J_{22}(y)\dy\}\int_{g(t)}^{h(t)}\big[u(t,x)+v(t,x)\big]\dx.
\eess
Thus we derive
\bes\label{3.1}
\int_{0}^{\yy}\int_{g(s)}^{h(s)}\big[u(s,x)+v(s,x)\big]\dx{\rm d}s\le\frac{L^*-h_0}{\min\{\mu_1\dd\int_{L^*}^{\yy}J_{11}(y)\dy, \mu_2\dd\int_{L^*}^{\yy}J_{22}(y)\dy\}}.
\ees
Simple computations as in Lemma \ref{l3.1} yield
\bess
&&\frac{d}{dt}\int_{g(t)}^{h(t)}\left[u(t,x)+\frac{c}{b}v(t,x)\right]\dx\ge-\max\{\frac{d_1}{\mu_1},\frac{cd_2}{b\mu_2}\}(h'(t)-g'(t))-\int_{g(t)}^{h(t)}\big[au+cv\big]\dx.
\eess
Integrating the above inequality from $0$ to $t$ leads to
\bess
&&\int_{g(t)}^{h(t)}\left[u+\frac{c}{b}v\right]\dx+\int_{0}^{t}\int_{g(s)}^{h(s)}\big[au(s,x)+cv(s,x)\big]\dx{\rm d}s\\
&&\ge\tau\int_{-h_0}^{h_0}\left[\vartheta_1(x)+\frac{c}{b}\vartheta_2(x)\right]\dx-\max\{\frac{d_1}{\mu_1},\frac{cd_2}{b\mu_2}\}L^*.
\eess
Notice that $\lim_{t\to\yy}\|u(t,x)+v(t,x)\|_{C([g(t),h(t)])}=0$. Using \eqref{3.1} and letting $t\to\yy$ in the above inequality arrives at
\bess
\tau\int_{-h_0}^{h_0}\left[\vartheta_1(x)+\frac{c}{b}\vartheta_2(x)\right]\dx-\max\{\frac{d_1}{\mu_1},\frac{cd_2}{b\mu_2}\}L^*\le\frac{(a+c)(L^*-h_0)}{\min\{\mu_1\dd\int_{L^*}^{\yy}J_{11}(y)\dy, \mu_2\dd\int_{L^*}^{\yy}J_{22}(y)\dy\}}.
\eess
This is a contradiction since the right side of the above inequality does not rely on $\tau$. Therefore, the proof is finished.
\end{proof}

Let the initial function $(u_0,v_0)=\tau(\vartheta_1,\vartheta_2)$ with $\tau>0$ and $(\vartheta_1,\vartheta_2)$ satisfying {\bf(I)}. Then combing Remark \ref{r3.1} with Lemma \ref{l3.6}, we can derive a critical value for $\tau$ governing spreading and vanishing.

\begin{lemma}\label{l3.7}Suppose that $2h_0<L^*$ and $J_{ii}>0$ in $\mathbb{R}$ for $i=1,2$. Then there exists a unique $\tau^*>0$ such that spreading happens if and inly if $\tau>\tau^*$.
\end{lemma}
\begin{proof}
Due to Remark \ref{r3.1}, we have that vanishing occurs for all small $\tau$. From comparison principle, it follows that the unique solution $(u,v,g,h)$ of \eqref{1.7} is strictly increasing in $\tau$. Then in view of Lemma \ref{l3.6}, we have that spreading happens for all large $\tau>0$. Hence we define
\[\tau^*=\inf\{\tau_0>0: {\rm spreading ~ happens ~ for ~ }\tau\ge\tau_0\}.\]
Clearly, $\tau^*$ is well defined and $\tau^*>0$. It is easy to see that spreading happens for all $\tau>\tau^*$. Now we show that vanishing occurs if $\tau<\tau^*$. Otherwise, there exists a $\tau_0\in (0,\tau^*)$ such that spreading happens with $\tau=\tau_0$. By the monotonicity of the unique solution of $(u,v,g,h)$ of \eqref{1.7}, we deduce that spreading occurs for all $\tau\ge\tau_0$ which obviously contradicts the definition of $\tau^*$. Thus if $\tau<\tau^*$, then vanishing must happen.

It remains to check the case $\tau=\tau^*$. If spreading happens for $\tau=\tau^*$, then there exists a $t_0>0$ such that $h(t_0)-g(t_0)>L^*$. By the continuous dependence of $(u,v,g,h)$ on $\tau$, there is a small $\ep>0$ such that the unique solution $(u_{\ep},v_{\ep},g_{\ep},h_{\ep})$ of \eqref{1.7} with $\tau=\tau^*-\ep$ satisfies that $h_{\ep}(t_0)-g_{\ep}(t_0)>L^*$ which, by Lemma \ref{l3.4}, implies that spreading occurs. Clearly, this is a contradiction. Hence if $\tau=\tau^*$, then vanishing ocuurs. The proof is finished.
\end{proof}
\begin{lemma}\label{l3.8}If $2h_0<L^*$, then there exists a $\underline{\mu}_1>0$ $(\underline{\mu}_2>0)$ which is independent of $\mu_2$ $(\mu_1)$ such that spreading happens if $\mu_1\ge\underline{\mu}_1$ $(\mu_2\ge\underline{\mu}_2)$.
\end{lemma}
\begin{proof}We only prove the assertion about $\mu_1$ since the similar method can be adopt for $\mu_2$.
Let $(\underline{u},\underline{v},\underline{h})$ be the unique solution of \eqref{1.7} with $\mu_2=0$. Clearly, $(\underline{u},\underline{v},\underline{h})$ is an lower solution of \eqref{1.7} and Lemmas \ref{l3.1}-\ref{l3.4} hold for $(\underline{u},\underline{v},\underline{h})$. Then we can argue as in the proof of \cite[Theorem 1.3]{DN4} to deduce that there exists a $\underline{\mu}_1>0$ such that spreading happens for $(\underline{u},\underline{v},\underline{h})$ if $\mu_1\ge\underline\mu_1$. Thus the assertion of $\mu_1$ is proved.
\end{proof}
The above lemma implies that if $\mu_1+\mu_2\ge\underline{\mu}_1+\underline{\mu}_2$, spreading will happen. Moreover, Lemma \ref{l3.6} shows vanishing will occur if $\mu_1+\mu_2\le \bar{\mu}$. However, based on these results we can not derive a critical value of $\mu_1+\mu_2$ such that spreading happens if and only if $\mu_1+\mu_2$ is beyond this critical value. The reason is that the unique solution $(u,v,h)$ is not monotone about $\mu_1+\mu_2$. But if we assume that there exists a function $f\in C([0,\yy))$ satisfying $f(0)=0$ and strictly increasing to $\yy$ such that $\mu_2=f(\mu_1)$ or $\mu_1=f(\mu_2)$, then we can obtain a unique critical value as we mentioned early. We only handle the case $\mu_2=f(\mu_1)$ since the other case is parallel.

\begin{lemma}\label{l3.9}Suppose $2h_0<L^*$ and $\mu_2=f(\mu_1)$ with $f$ defined as above. Then there exists a unique $\mu^*_1>0$ such that spreading happens if and only if $\mu_1>\mu^*_1$.
\end{lemma}
\begin{proof}
By Lemma \ref{l3.5}, vanishing happens if $\mu_1+f(\mu_1)\le\bar{\mu}$. Due to the assumptions on $f$, there exists a unique $\tilde\mu_1>0$ such that $\tilde\mu_1+f(\tilde\mu_1)=\bar{\mu}$ and $(\mu_1+f(\mu_1)-\bar{\mu})(\mu_1-\tilde{\mu}_1)>0$ if $\mu_1\neq\tilde{\mu}_1$. Hence $\mu_1+f(\mu_1)\le\bar{\mu}$ is equivalent to $\mu_1\le\tilde{\mu}_1$. By lemma \ref{l3.6}, if $\mu_1>\underline{\mu}_1$, spreading occurs. Then we can use the monotonicity of $(u,v,g,h)$ on $\mu_1$ and argue as in the proof of \cite[Theorem 3.14]{CDLL} to complete the proof. The details are omitted.
\end{proof}

Next we consider the effect of diffusion coefficients $d_1$ and $d_2$ on criteria for spreading and vanishing. Assume $c=G'(0)$ and $J_{12}=J_{21}$. Then thanks to Propositions \ref{p2.1} and \ref{p2.3}, the variational characteristic holds and the monotonicity of $\lambda_p$ on diffusion coefficients $d_1$ and $d_2$ is valid. For clarity, we rewrite $\lambda_p(l_1,l_2)$ as $\lambda_p(l_1,l_2,d_1,d_2)$. By Proposition \ref{p2.3}, for all $(l_1,l_2)$, $\lambda_p(l_1,l_2,d_1,d_2)\to-\yy$ as $(d_1,d_2)\to(\yy,\yy)$, while $\lambda_p(l_1,l_2,d_1,d_2)\to\lambda_p(l_1,l_2,0,0)$ as $(d_1,d_2)\to(0,0)$. Moreover, in view of Proposition \ref{p2.2}, we see $\lambda_p(l_1,l_2,0,0)\to\max\{-a,-b\}$ as $l_2-l_1\to0$ and $\lambda_p(l_1,l_2,0,0)\to\lambda_A>0$ ($\mathcal{R}_0>1$) as $l_2-l_1\to\yy$. Hence there exists a unique $\tilde{L}^*$ depending only on $(a,b,c,J_{12})$ such that $\lambda_p(l_1,l_2,0,0)=0$ if $l_2-l_1=\tilde{L}^*$ and $\lambda_p(l_1,l_2,0,0)(l_2-l_1-\tilde{L}^*)>0$ if $l_2-l_1\neq \tilde{L}^*$.

As is seen from Lemma \ref{l3.4}, the critical value $L^*$ depends on $d_1$ and $d_2$, while $\tilde{L}^*$ only relys on $(a,b,c,J_{12})$. Moreover, since $\lambda_p(l_1,l_2,d_1,d_2)<\lambda_p(l_1,l_2,0,0)$ , by Proposition \ref{p2.2} we easily derive $L^*>\tilde{L}^*$ for all $d_i>0$ with $i=1,2$. Then we have the following result.

\begin{lemma}\label{l3.10}Assume that $c=G'(0)$, $J_{12}=J_{21}$ and $d_2=f(d_1)$ with $f$ defined as in Lemma \ref{l3.9}. Then if $2h_0>\tilde{L}^*$, then we can find a unique $d^*_1$ such that spreading happens if $d_1\ge d^*$, while if $d_1>d^*_1$, whether spreading or vanishing occurs depends on $\mu_i$ for $i=1,2$ as in Lemmas \ref{l3.5} and \ref{l3.8}. If $2h_0\le \tilde{L}^*$, then for any $d_1>0$, spreading and vanishing both may happen, depending on $\mu_i$ for $i=1,2$ as in Lemmas \ref{l3.5} and \ref{l3.8}.
\end{lemma}
\begin{proof}For convenience, we denote $\lambda_p(-h_0,h_0,d_1,f(d_1))$ by $\lambda_p(h_0,d_1)$.
Owing to our assumptions and Proposition \ref{l2.3}, we know that $\lambda_p(h_0,d_1)$ is strictly decreasing and continuous in $d_1\ge0$. Thanks to $2h_0>\tilde{L}^*$, we have $\lambda_p(h_0,d_1)\to\lambda_p(h_0,0)>0$ as $d_1\to0$. Notice that $\lambda_p(h_0,d_1)\to-\yy$ as $d_1\to\yy$. Thus there exists a unique $d^*_1$ such that $\lambda_p(h_0,d_1)\ge0$ if $d_1\le d^*_1$ which, combined with Lemma \ref{l3.1}, yields that spreading occurs. If $d_1> d^*_1$, then $\lambda_p(h_0,d_1)<0$. We can argue as in the proof of Lemmas \ref{l3.5} and \ref{l3.8} to derive the desired result. If $2h_0\le\tilde{L}^*$, then $\lambda_p(h_0,d_1)<0$ for all $d_1>0$. So similarly, we can obtain the result as wanted. The proof is complete.
\end{proof}

\begin{remark}\label{r3.2}Under the assumptions of Lemma \ref{l3.10}, we rewrite $L^*$ as $L^*(d_1)$ to stress the dependence on $d_1$. By Proposition \ref{p2.3}, we easily find that $L^*(d_1)$ is continuous and strictly increasing in $d_1>0$, $L^*(d_1)\to \tilde{L}^*$ as $d_1\to0$, and $L^*(d_1)\to\yy$ as $d_1\to\yy$.
\end{remark}

\begin{remark}\label{r3.3} From \cite[Theorem 1.3]{CDLL}, \cite[Theorem 1.3]{ZZLD}, \cite[Theorem 1.3]{CD}, \cite[Theorem 1.3]{DN4} and \cite[Theorem 1.2]{NV}, we can find that if diffusion coefficients are small enough, then spreading will happen no matter what other parameters are. However, for our model, Lemma \ref{l3.10} indicates that if $h_0$ is sufficiently small, then for all diffusion coefficients, both spreading and vanishing may happen. This phenomenon may imply that the nonlocal reaction makes spreading difficult.
\end{remark}

\section{A discussion on the nonlocal reaction term}
In this short section, we discuss the effect of nonlocal reaction term on spreading and vanishing. As is seen from \cite[Theorem 1.3]{CDLL}, \cite[Theorem 1.3]{ZZLD}, \cite[Theorem 1.3]{CD}, \cite[Theorem 1.3]{DN4} and \cite[Theorem 1.2]{NV}, the critical length of habitat, determined by a related principal eigenvalue problem, plays a crucial role in criteria for spreading and vanishing. Now by comparing the critical lengths in models \eqref{1.3}-\eqref{1.7}, we intend to figure out whether the nonlocal reaction term makes the spreading of an epidemic difficult. In other words, for models \eqref{1.3}-\eqref{1.7}, we conjecture that the more nonlocal reaction terms a model has, the harder spreading happens for the model.

For clarity, we list the corresponding principal eigenvalue problems for models \eqref{1.3}-\eqref{1.6}, respectively.
\bes\label{4.1}d\int_{0}^{l}J(x-y)\phi(y)\dy-d\phi-a\phi+\frac{cG'(0)}{b}\phi=\lambda\phi, ~ ~ x\in[0,l].\ees
\bes\label{4.2}
d\int_{0}^{l}J(x-y)\phi(y)\dy-d\phi-a\phi+\frac{cG'(0)}{b}\int_{0}^{l}K(x-y)\phi(y)\dy=\lambda\phi, ~ ~ x\in[0,l].
\ees
\bes\label{4.3}
\left\{\begin{aligned}
&d_1\int_{0}^{l}J_{1}(x-y)\varphi_1(y)\dy-d_1\varphi_1-a\varphi_1+c\varphi_2=\lambda\varphi_1, & &x\in[0,l],\\[1mm]
&d_2\int_{0}^{l}J_{2}(x-y)\varphi_2(y)\dy-d_{2}\varphi_2-b\varphi_2+G'(0)\varphi_1=\lambda\varphi_2, & &x\in[0,l].
 \end{aligned}\right.
\ees
\bes\label{4.4}
\left\{\begin{aligned}
&d_1\int_{0}^{l}J_{1}(x-y)\varphi_1(y)\dy-d_1\varphi_1-a\varphi_1+c\int_{0}^{l}K(x-y)\varphi_2(y)\dy=\lambda\varphi_1, & &x\in[0,l],\\[1mm]
&d_2\int_{0}^{l}J_{2}(x-y)\varphi_2(y)\dy-d_{2}\varphi_2-b\varphi_2+G'(0)\varphi_1=\lambda\varphi_2, & &x\in[0,l].
 \end{aligned}\right.
\ees
It is noted that for convenience, we consider the above principal eigenvalue problems on interval $[0,l]$ since these principal eigenvalues depend only on the length of an interval.
Denote their principal eigenvalues by $\lambda_i(l)$ with $i=1,2,3,4$, respectively. Clearly, for these four principal eigenvalue problems, there exist the unique critical lengths, denoted respectively by $L^*_1, L^*_2, L^*_3,L^*_4$, such that $\lambda_i(L^*_i)=0$. Let $(\nu(l),\omega)$ be the principal eigenpair of
\[\int_{0}^{l}J(x-y)\omega(y)\dy-\omega=\nu\omega, ~ ~ x\in[0,l].\]
It is well known that $\nu(l)$ is strictly increasing in $l>0$, $\lim_{l\to0}\nu(l)=-1$ and $\lim_{l\to\yy}\nu(l)=0$.

In the following, we assume that all kernel functions in \eqref{4.1}-\eqref{4.4} are the same and represented by $J$, i.e. $J_1=J_2=K=J$. Then using the properties of principal eigenvalue and some simple calculations, we have $\lambda_1(l)=d\nu(l)-a+\frac{cG'(0)}{b}$ and $\lambda_2(l)=d\nu(l)-a+\frac{cG'(0)}{b}(\nu(l)+1)$. Since $\nu(l)\in(-1,0)$ for all $l>0$, we derive $\lambda_1(l)>\lambda_2(l)$ which, combined with the monotonicity of $\lambda_i(l)$ on $l$, implies $L^*_1<L^*_2$.

For problems \eqref{4.3} and \eqref{4.4}, owing to our assumption $J_1=J_2=K=J$, we easily see that these two principal eigenvalue problems can be transformed to the following two algebraic eigenvalue problems respectively.
 \bess
\left\{\begin{aligned}
&d_1\nu(l) p_1-ap_1+cp_2=\lambda p_1,\\[1mm]
&d_2\nu(l) p_2-bp_2+G'(0)p_1=\lambda p_2, ~ ~ {\rm and } ~ ~
 \end{aligned}\right.
 \left\{\begin{aligned}
&d_1\nu(l) p_1-ap_1+c(\nu(l)+1)p_2=\lambda p_1,\\[1mm]
&d_2\nu(l) p_2-bp_2+G'(0)p_1=\lambda p_2.
 \end{aligned}\right.
\eess
Then straightforward computations show that when
\bess\left\{\begin{aligned}
&\lambda_3(l)=\frac{d_1\nu(l)-a+d_2\nu(l)-b+\sqrt{(d_2\nu(l)-b-d_1\nu(l)+a)^2+4cG'(0)}}{2},\\
&\lambda_4(l)=\frac{d_1\nu(l)-a+d_2\nu(l)-b+\sqrt{(d_2\nu(l)-b-d_1\nu(l)+a)^2+4cG'(0)(\nu(l)+1)}}{2},
\end{aligned}\right.
\eess
their corresponding eigenvectors $(p_1,p_2)^T$ are positive.
Recall $\nu(l)\in(0,1)$. So $\lambda_3(l)>\lambda_4(l)$ which, together with the monotonicity again, directly yields $L^*_3<L^*_4$.

Moreover, if we also let $J_{ij}=J$ for $i,j=1,2$ in \eqref{2.9}, then \eqref{2.9} can reduce to
 \bess
 \left\{\begin{aligned}
&d_1\nu(l) p_1-ap_1+c(\nu(l)+1)p_2=\lambda p_1,\\[1mm]
&d_2\nu(l) p_2-bp_2+G'(0)(\nu(l)+1)p_1=\lambda p_2.
 \end{aligned}\right.
\eess
Analogously, it is easy to see that the principal eigenvalue $\lambda_p(l)$ of \eqref{2.9} takes form of
\bess
\lambda_p(l)=\frac{d_1\nu(l)-a+d_2\nu(l)-b+\sqrt{(d_2\nu(l)-b-d_1\nu(l)+a)^2+4cG'(0)(\nu(l)+1)^2}}{2}.
\eess
Clearly, $\lambda_4(l)>\lambda_p(l)$. Thus $L^*_4<L^*$ where $L^*$ is defined as in Lemma \ref{l3.4}, i.e. $\lambda_p(L^*)=0$.

All in all, if we assume that all the kernel functions appearing in principal eigenvalue problems \eqref{4.1}-\eqref{4.4} and \eqref{2.9} are the same and denoted by $J$, then we have $\lambda_1(l)>\lambda_2(l)$ and $\lambda_3(l)>\lambda_4(l)>\lambda_p(l)$, which together with the monotonicity of principal eigenvalue on $l$, leads to $L^*_1<L^*_2$ and $L^*_3<L^*_4<L^*$. This indicates that the more nonlocal reaction terms a model has, the larger its critical length for the initial habitat is, making the spreading of an epidemic more difficult to happen.

However, we point out that the extra assumption that all kernel functions are the same is a technical condition, by which we can compare the above principal eigenvalues.


\begin{thebibliography}{99}
\bibliographystyle{siam}
\setlength{\baselineskip}{15pt}
\vspace{-1.5mm}\bibitem{CP} V. Capasso and S.L. Paveri-Fontana, {\it A mathematical model for the 1973 cholera epidemic in the European Mediterranean region}, Rev. Epidemiol. Sante Publique, \textbf{27} (1979), 121-132.

\vspace{-1.5mm}\bibitem{CM}V. Capasso and L. Maddalena, {\it Convergence to equilibrium states for a reaction-diffusion system modeling the spatial spread of a class of bacterial and viral diseases}, J. Math. Biol., \textbf{13} (1981), 173–184.

\vspace{-1.5mm}\bibitem{ZW} X.-Q. Zhao and W.D. Wang, {\it Fisher waves in an epidemic model}, Discrete Contin. Dyn. Syst., Ser. B, \textbf{4} (2004), 1117–1128.

\vspace{-1.5mm}\bibitem{WSL}S.L. Wu, Y.J. Sun and S.Y. Liu, {\it Traveling fronts and entire solutions in partially degenerate reaction-diffusion systems with monostable nonlinearity}, Discrete Contin. Dyn. Syst., \textbf{33} (2013), 921–946.

\vspace{-1.5mm}\bibitem{ABL}I. Ahn, S. Beak and Z.G. Lin, {\it The spreading fronts of an infective environment in a man-environment-man epidemic model}, Appl. Math. Model., \textbf{40} (2016), 7082–7101.

\vspace{-1.5mm}\bibitem{WC}J. Wang and J.-F. Cao {\it The spreading frontiers in partially degenerate reaction–diffusion systems}, Nonlinear Anal., \textbf{122} (2015), 215–238.

\vspace{-1.5mm}\bibitem{LZW}W.-T. Li, M. Zhao and J. Wang, {\it Spreading fronts in a partially degenerate integro-differential reaction–diffusion system}, Z. Angew. Math. Phys, (2017), 68:109.

\vspace{-1.5mm}\bibitem{ZLN}M. Zhao, W.-T. Li and W.J. Ni, {\it Spreading speed of a degenerate and cooperative epidemic model with free boundaries}, Discrete Contin. Dyn. Syst., Ser. B, \textbf{25} (2020), 981–999.

\vspace{-1.5mm}\bibitem{CDLL} J.-F. Cao, Y.H. Du, F. Li and W.-T. Li, {\it The dynamics of a Fisher-KPP nonlocal diffusion model with free boundaries},  J. Funct. Anal., \textbf{277} (2019), 2772-2814.

\vspace{-1.5mm}\bibitem{ZZLD}M. Zhao, Y. Zhang, W.-T. Li and Y.H. Du, {\it The dynamics of a degenerate epidemic model with nonlocal diffusion and free boundaries}, J. Differential Equations, \textbf{269} (2020), 3347-3386.

\vspace{-1.5mm}\bibitem{DN1}Y.H. Du and W.J. Ni, {\it Spreading speed for some cooperative systems with nonlocal diffusion and free boundaries, part 1: Semi-wave and a threshold condition}, J. Differential Equations, \textbf{308} (2022), 369-420.

\vspace{-1.5mm}\bibitem{DLZJMPA}Y.H. Du, F. Li and M.L. Zhou, {\it Semi-wave and spreading speed of the nonlocal Fisher-KPP equation with free boundaries}, J. Math. Pures Appl., \textbf{154} (2021), 30-66.

\vspace{-1.5mm}\bibitem{ZLD}M. Zhao, W.-T. Li and Y.H. Du, {\it The effect of nonlocal reaction in an epidemic model with nonlocal diffusion and free boundaries}, Comm. Pure Appl. Anal., \textbf{19} (2020), 4599-4620.

\vspace{-1.5mm}\bibitem{DLNZ}Y.H. Du, W.-T. Li, W.J. Ni and M. Zhao, {\it Finite or infinite spreading speed of an epidemic model with free boundary and double nonlocal effects}, J. Dyn. Diff. Equat., (2022), https://doi.org/10.1007/s10884-022-10170-1.

\vspace{-1.5mm}\bibitem{CD}T.-Y. Chang and Y.H. Du, {\it Long-time dynamcis of an epidemic model with nonlocal diffusion and free boundaries}, Electronic Research Archive, \textbf{30} (2022), 289-313.

\vspace{-1.5mm}\bibitem{WD1}R. Wang and Y.H. Du, {\it Long-time dynamics of a nonlocal epidemic model with free boundaries: spreading-vanishing dichotomy}, J. Differential Equations, \textbf{327} (2022), 322-381.

\vspace{-1.5mm}\bibitem{WD2}R. Wang and Y.H. Du, {\it Long-time dynamics of a nonlocal epidemic model with free boundaries: spreading speed}, Discrete. Contin. Dyn. Syst.,  \textbf{43} (2023), 121-161.

\vspace{-1.5mm}\bibitem{DL}Y.H. Du and Z.G. Lin, {\it Spreading-vanishing dichotomy in the diffusive logistic model with a free boundary}, SIAM J. Math. Anal., \textbf{42} (2010), 377-405.

\vspace{-1.5mm}\bibitem{CQW}C. Cort{\'a}zar, F. Quir{\'o}s and N. Wolanski, {\it A nonlocal diffusion problem with a sharp free boundary}, Interfaces Free Bound., \textbf{21} (2019), 441-462.

\vspace{-1.5mm}\bibitem{DN2}Y.H. Du and W.J. Ni, {\it Rate of propagation for the Fisher-KPP equation with nonlocal diffusion and free boundaries}, J. Eur. Math. Soc., (2023), Doi:10.4171/JEMS/1392, pp. 1–53.

\vspace{-1.5mm}\bibitem{DN3}Y.H. Du and W.J. Ni, {\it Exact rate of accelerated propagation in the Fisher-KPP equation with nonlocal diffusion and free boundaries}, Math. Ann., \textbf{389} (2024), 2931-2958.

\vspace{-1.5mm}\bibitem{DN4}Y.H. Du and W.J. Ni, {\it Analysis of a West Nile virus model with nonlocal diffusion and free boundaries}, Nonlinearity, \textbf{33} (2020), 4407-4448.

\vspace{-1.5mm}\bibitem{NV} T.-H. Nguyen and H.-H. Vo, {\it Dynamics for a two-phase free boundaries system in an epidemiological model with couple nonlocal dispersals}, J. Differential Equations, \textbf{335} (2022), 398-463.

\vspace{-1.5mm}\bibitem{PLL} L.Q. Pu, Z.G. Lin and Y. Lou, {\it A West Nile virus nonlocal model with free boundaries and seasonal succession}, J. Math. Biol., \textbf{86} (2023), 25.

\vspace{-1.5mm}\bibitem{DWZ}Y.H. Du, M.X. Wang and M. Zhao, {\it Two species nonlocal diffusion systems with free boundaries}, Discrete Contin. Dyn. Syst., \textbf{42} (2022), 1127-1162.

\vspace{-1.5mm}\bibitem{DN5}Y.H. Du and W.J. Ni, {\it The high dimensional Fisher-KPP nonlocal diffusion equation with free boundary and radial symmetry, Part 1}, SIAM J. Math. Anal., \textbf{54} (2022), 3930-3973.

\vspace{-1.5mm}\bibitem{ZLZ}W.Y. Zhang, Z.H. Liu and L. Zhou, {\it Dynamics of a nonlocal diffusive logistic model with free boundaries in time periodic environment}, Discrete Contin. Dyn. Syst. B., \textbf{26} (2021), 3767-3784.

\vspace{-1.5mm}\bibitem{LLW1}L. Li, W.-T. Li and M.X. Wang, {\it Dynamics for nonlocal diffusion problems with a free boundary}, J. Differential Equations, \textbf{330} (2022), 110-149.

\vspace{-1.5mm}\bibitem{SWZ}Y.-H. Su, X.F. Wang and T. Zhang, {\it Principal spectral theory and variational characterizations for cooperative systems with nonlocal and coupled diffusion}, J. Differential Equations, \textbf{369} (2023), 94-114.

\vspace{-1.5mm}\bibitem{LCW}F. Li, J. Coville and X.F. Wang, {\it On eigenvalue problems arising from nonlocal diffusion models}, Discrete. Contin. Dyn. Syst.,  \textbf{37}  (2017), 879-903.

\vspace{-1.5mm}\bibitem{HMMV}V. Hutson, S. Martinez, K. Mischaikow and G.T. Vickers, {\it The evolution of dispersal}, J. Math. Biol., \textbf{47} (2003), 483-517.

\vspace{-1.5mm}\bibitem{SLLW}Y.-H. Su, W.-T. Li, Y. Lou and X.F. Wang, {\it Principal spectral theory for nonlocal systems and applicationss to stem cell regeneration models}, J. Math. Pures Appl., \textbf{176} (2023), 226-281.

\vspace{-1.5mm}\bibitem{Zhanglei}L. Zhang, {\it Principal spectral theory and asymptotic behavior of the spectral bound for partially degenerate nonlocal dispersal systems}, preprint, (2023), arXiv: 2307.16221v2.

\vspace{-1.5mm}\bibitem{Bu}R. B{\"u}rger, {\it Perturbations of positive semigroups and applications to population genetics}, Math. Z., \textbf{197} (1988), 259-272.
\end{thebibliography}
\end{document}